\documentclass{amsart}
\usepackage[margin = 3 cm,marginpar= 2cm]{geometry}
\usepackage{mathtools,xcolor,enumitem,subcaption,amssymb,mathrsfs}

\numberwithin{equation}{section} 

\usepackage{mleftright} 

\newcommand*{\K}{\mathbb{K}}
\newcommand*{\M}{\mathbb{M}}

\newcommand*{\R}{\mathbb{R}}

\newcommand*{\Fcal}{\mathcal{F}}

\newcommand*{\Jcal}{\mathcal{J}}

\newcommand*{\Scal}{\mathcal{S}}
\newcommand*{\Tcal}{\mathcal{T}}
\newcommand*{\Ucal}{\mathcal{U}}
\newcommand*{\Vcal}{\mathcal{V}}
\newcommand*{\Wcal}{\mathcal{W}}

\newcommand*{\Cscr}{\mathscr{C}}
\newcommand*{\Dscr}{\mathscr{D}}
\newcommand*{\Hscr}{\mathscr{H}}

\newcommand*{\Kscr}{\mathscr{K}}

\newcommand*{\Mscr}{\mathscr{M}}
\newcommand*{\Nscr}{\mathscr{N}}
\newcommand*{\Pscr}{\mathscr{P}}

\newcommand*{\Tscr}{\mathscr{T}}

\newcommand*{\Xscr}{\mathscr{X}}

\newcommand*{\Lp}[1]{L^{#1}}

\newcommand{\Proj}[1]{\mathop{}\mathrm{P}_{#1}}

\newcommand{\dee}{\mathop{}\!\mathrm{d}}
\newcommand*{\e}{\mathrm{e}}

\DeclareMathOperator*{\sgn}{sgn}

\DeclareMathOperator*{\dom}{dom}
\DeclareMathOperator*{\ran}{ran}

\newcommand*{\Id}{\mathrm{Id}}
\newcommand*{\del}{\partial}

\newcommand*{\phiv}{\varphi}

\newtheorem{thm}{Theorem}[section]
\newtheorem{cor}[thm]{Corollary}
\newtheorem{lem}[thm]{Lemma}
\newtheorem{prp}[thm]{Proposition}
\newtheorem{res}{Result}

\theoremstyle{remark}
\newtheorem{rem}[thm]{Remark}

\theoremstyle{definition}
\newtheorem{dfn}[thm]{Definition}

\renewcommand{\left}{\mleft}
\renewcommand{\right}{\mright}

\let\bs\boldsymbol

\newcommand{\ip}[2]{\langle#1,#2\rangle}

\DeclarePairedDelimiter\abs{\lvert}{\rvert}
\DeclarePairedDelimiter\norm{\lVert}{\rVert}%

\newcommand{\diff}[2]{\frac{\dee #1}{\dee #2}}

\def\Xint#1{\mathchoice
  {\XXint\displaystyle\textstyle{#1}}%
  {\XXint\textstyle\scriptstyle{#1}}%
  {\XXint\scriptstyle\scriptscriptstyle{#1}}%
  {\XXint\scriptscriptstyle\scriptscriptstyle{#1}}%
  \!\int}
\def\XXint#1#2#3{{\setbox0=\hbox{$#1{#2#3}{\int}$}
    \vcenter{\hbox{$#2#3$}}\kern-.5\wd0}}
\def\dint{\Xint-}

\allowdisplaybreaks 
\usepackage{hyperref}
\newcommand{\arxiv}[1]{\href{https://arxiv.org/abs/#1}{arXiv:#1}} 

\title[The sticky dynamics of the 1D Euler-alignment system as a gradient flow]{The sticky particle dynamics of the 1D pressureless Euler-alignment system as a gradient flow}

\author{Sondre Tesdal Galtung}
\address{Department of Mathematical Sciences, NTNU -- Norwegian University of Science and Technology, 7491 Trondheim, Norway}
\email{sondre.galtung@ntnu.no}


\begin{document}
  
  \begin{abstract}
    We show how the sticky dynamics for the one-dimensional pressureless Euler-alignment system can be obtained as an \(L^2\)-gradient flow of a convex functional.
    This is analogous to the Lagrangian evolution introduced by Natile and Savar\'{e} for the pressureless Euler system, and by Brenier et al.\ for the corresponding system with a self-interacting force field.

    Our Lagrangian evolution can be seen as the limit of sticky particle Cucker--Smale dynamics, similar to the solutions obtained by Leslie and Tan from a corresponding scalar balance law, and provides us with a uniquely determined distributional solution of the original system in the space of probability measures with quadratic moments and corresponding square-integrable velocities.
    Moreover, we show that the gradient flow also provides an entropy solution to the balance law of Leslie and Tan, and how their results on cluster formation follow naturally from (non-)monotonicity properties of the so-called natural velocity of the flow.
  \end{abstract}
  
  \maketitle

  \section{Introduction}
  We consider the one-dimensional pressureless Euler-alignment system for density \(\rho_t\) and scalar velocity \(v_t\) depending on time \(t \ge 0\) and position \(x \in \R\), which reads
  \begin{subequations}\label{eq:EA}
    \begin{align}
      \del_t \rho_t + \del_x(\rho_t v_t) &= 0, \label{eq:EA:den} \\
      \del_t(\rho_t v_t) + \del_x(\rho_t v^2_t) &= \rho_t (\phi\ast(\rho_t v_t)) - \rho_t v_t (\phi\ast\rho_t). \label{eq:EA:mom}
    \end{align}
  \end{subequations}
  Here \(\phi\) is a communication protocol appearing in convolutions with the density and momentum,
  which we assume to be symmetric, radially nonincreasing and locally integrable, i.e.,  \(\phi \in \Lp{1}_{\mathrm{loc}}(\R)\).
  In particular, \(\phi\) may be \textit{weakly singular} of order \(\beta \in (0,1)\), that is, there exist constants \(R > 0\) and \(c > 0\) such that
  \begin{equation}\label{eq:w-sing}
    \phi(r) \ge c r^{-\beta} \quad \text{for all } r \in (0,R).
  \end{equation}
  A different facet of the communication protocol concerns its asymptotic behavior.
  Here \(\phi\) can have a thin or fat tail, respectively defined by
  \begin{equation}\label{eq:tail}
    \int_1^\infty \phi(x)\dee x < \infty \qquad \text{or} \qquad \int_1^\infty \phi(x)\dee x = \infty.
  \end{equation}
  That is, the tail of \(\phi\) is either integrable or not.
  A thin tail implies that there is only weak or no communication over larger distances, while a fat tail means there is strong communication.
  
  The Euler-alignment system originates from the theory of collective behavior, see \cite{shvydkoy2021dynamics} for a comprehensive overview of such models.
  It has been studied under various regularity assumptions, and one may even consider measure-valued solutions.
  A formal integration of \eqref{eq:EA} shows that both the total mass and momentum are conserved, the latter by virtue of the symmetry of the communication protocol.
  
  The corresponding particle dynamics, which can be obtained by considering \eqref{eq:EA} for the empirical measures
  \begin{equation*}
    \rho_t = \sum_{i=1}^{N}m_i \delta(x-x_i), \qquad \rho_t v_t = \sum_{i=1}^{N} m_i v_i \delta(x-x_i),
  \end{equation*}
  where \(\delta\) denotes a Dirac measure, is described by the following system of ordinary differential equations,
  \begin{equation}\label{eq:EAparticles}
    \dot{x}_i = v_i, \qquad \dot{v}_i = -\sum_{\substack{j=1 \\ x_j\neq x_i}}^{N} m_j \phi(x_i-x_j)(v_i-v_j)
  \end{equation}
  for masses \(\{m_i\}_{i=1}^N\), positions \(\{x_i\}_{i=1}^N\) and velocities \(\{v_i\}_{i=1}^N\) of \(N\) particles.
  This corresponds to the well-known Cucker--Smale model for swarming \cite{cucker2007emergent}.
  
  Such alignment models, especially those with Cucker--Smale-type communication protocols, have been studied intensively in both one and more dimensions,
  covering both hydrodynamic and particle systems like \eqref{eq:EA} and \eqref{eq:EAparticles}, as well as kinetic equations \cite{choi2021one,peszek2023heterogeneous}.
  Topics of interest include well-posedness of solutions and critical thresholds for blow-up \cite{tadmor2014critical,carrillo2017sharp,bhatnagar2023critical}, and asymptotic behavior \cite{ha2018first,ha2019complete,lear2022geometric}.
  An important notion in this context is flocking, that the support of the solutions remain bounded in time; it has been shown in \cite{ha2009simple,tadmor2014critical} that fat-tailed protocols lead to flocking for both the particle system \eqref{eq:EAparticles} and strong solutions of \eqref{eq:EA}.
  
  When the communication protocol is too singular, i.e., not locally integrable, concentration of mass cannot happen, cf.\ \cite{carrillo2017sharp}.
  Since we are specifically interested in mass which may collide and remain concentrated, we will here assume \(\phi \in \Lp{1}_\text{loc}(\R)\).
  Moreover, as detailed below, we will study the connection between \eqref{eq:EA} and \eqref{eq:EAparticles} through a gradient flow structure. On this note, we mention recent the works \cite{peszek2023heterogeneous,peszek2022measure} which concern a gradient flow framework for kinetic alignment systems.
  
  \subsection{Reduced systems}
  Here we give the ideas for how the second-order Euler-alignment system has been studied through reduction to first-order systems, both in the particle and continuum settings.
  
  Leveraging the fact that we are on the line, the alignment force driving the acceleration in \eqref{eq:EAparticles} can be written as an exact derivative, involving the primitive of the communication kernel \(\phi\),
  \begin{equation}\label{eq:Phi}
    \Phi(x) = \int_0^x \phi(y)\dee y.
  \end{equation}
  From our assumptions it follows that this is a continuous and odd function.
  An extreme case is the so-called all-to-all communication \(\phi\equiv K\) for some positive constant \(K > 0\), cf.\ \cite{amadori2021bv}, in which case \(\Phi(x) = Kx\).
  Figure \ref{fig:kernels} illustrates both bounded and weakly singular communication protocols \(\phi\), as well as their corresponding primitives \(\Phi\).
  \begin{figure}
  	\includegraphics[width=0.7\linewidth]{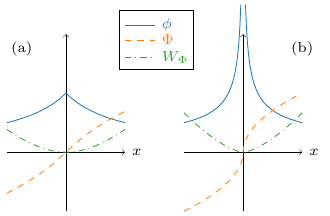}
  	\caption{Illustration of the function \(\Phi\) in the case of (a) bounded and (b) weakly singular communication protocols \(\phi\). We will return to the convex function \(W_{\Phi}\), the primitive of \(\Phi\), in Section \ref{s:Lag}.}
  	\label{fig:kernels}
  \end{figure}
  For future reference, we list some useful properties of \(\Phi\).
  \begin{lem}\label{lem:Phi:props}
    The function \(\Phi\) in \eqref{eq:Phi} has the following properties:
    \begin{enumerate}[label={\upshape(\alph*)}]
      \item It is continuous, nondecreasing and odd, \(\Phi(-x) = -\Phi(x)\).
      \item It is concave and subadditive for \(x \ge 0\), and so by (a), it is convex and superadditive for \(x \le 0\),
      \begin{equation*}
        \Phi(\abs{x}+\abs{y}) \le \Phi(\abs{x}) + \Phi(\abs{y}).
      \end{equation*}
      In particular, if \(\phi \neq 0\), then \(\Phi\) is a modulus of continuity, that is, a concave and continuous function \(\Phi \colon [0,\infty)\to[0,\infty)\) which satisfies \(0 = \Phi(0) < \Phi(r)\) for any \(r > 0\).
      \item It is pointwise linearly bounded, i.e.,
      \begin{equation}\label{eq:Phi:linbnd}
        \abs{\Phi(x)} \le \Phi(1)\max\{1,\abs{x}\} \le \Phi(1)(1+\abs{x}).
      \end{equation}
      \item By (a), (b) and the (reverse) triangle inequality,
      \begin{equation*}
        \abs{\abs{\Phi(x)}-\abs{\Phi(y)}} \le \Phi(\abs{\abs{x}-\abs{y}}) \le \Phi(\abs{x-y}).
      \end{equation*}
    \end{enumerate}
  \end{lem}
  \begin{proof}
    Property (a) follows from the definition \eqref{eq:Phi} and \(\phi \ge 0\).
    Property (b) is a consequence of \(\phi\) being nonincreasing for \(x \ge 0\). Since \(\phi(x)\) is nonincreasing for \(x \ge 0\), we furthermore have that \(\Phi(x)\) is concave and subadditive for \(x \ge 0\).
    For property (c), observe that for \(x\in[0,1]\), \(\Phi(x) \le \Phi(1) = \norm{\phi}_{\Lp{1}(0,1)}\). On the other hand, for \(x > 1\), \(\Phi(x)-\Phi(1) \le \Phi(1)(x-1) \) since the integrand \(\phi\) is nonincreasing and must be less than its average \(\Phi(1)\) on \([0,1]\). Hence, by symmetry, \(\abs{\Phi(x)} \le \Phi(1)\max\{1,\abs{x}\}\).
    {The second inequality in (d) is a consequence of \(\Phi\) being nondecreasing and the reverse triangle inequality \(\abs{\abs{x}-\abs{y}} \le \abs{x-y}\). The first inequality can be rephrased as \(\abs{\int_{\abs{y}}^{\abs{x}}\phi(z)\dee z} \le \int_{0}^{\abs{\abs{x}-\abs{y}}}\phi(z)\dee z\)}, which is true due to \(\phi\) being nonincreasing.
  \end{proof}
  
  The function \(\Phi\) in \eqref{eq:Phi} plays a central role in several works, e.g., \cite{ha2018first,leslie2020lagrangian} where one takes advantage of working on the line in order to reduce the second order system \eqref{eq:EAparticles} to a first-order particle system, and similarly in \cite{leslie2023sticky} where \eqref{eq:EA} is studied using a scalar balance law.
  
  \subsubsection{The particle case}
  In the discrete, or particle case, the initial value problem for \eqref{eq:EAparticles} can be rewritten as
  \begin{equation*}
    \diff{}{t} \left( \dot{x}_i + \sum_{j=1}^{N} m_j \Phi(x_i-x_j) \right) = 0, \quad x_i(0) = \bar{x}_i, \quad v_i(0) = \bar{v}_i,
  \end{equation*}
  which in turn can be integrated to yield the Kuramoto-type equation
  \begin{equation}\label{eq:CS}
    \dot{x}_i + \sum_{j=1}^{N} m_j \Phi(x_i-x_j) = \bar{v}_i + \sum_{j=1}^{N} m_j \Phi(\bar{x}_i-\bar{x}_j) \eqqcolon \bar{\psi}_i.
  \end{equation}
  Observe how in the reduction to a first-order system, \(\bar{v}_i\) goes from an initial condition to a parameter through the quantity \(\bar{\psi}_i\), known as the \textit{natural velocity}, cf.\ \cite{ha2018first}.
  This is likely in analogy to the corresponding quantity for the original Kuramoto system \cite{kuramoto1975self}, which models coupled, nonlinear oscillators and reads
  \begin{equation}\label{eq:kuramoto}
    \dot{\theta}_i = \omega_i - \frac{\nu}{N}\sum_{j=1}^{N}\sin(\theta_i - \theta_j), \quad i \in \{1,\dots,N\}, \quad \nu > 0.
  \end{equation}
  Here \(\theta_i\) is the phase of the \(i\)\textsuperscript{th} oscillator, while the constant \(\omega_i\) is often called its \textit{natural frequency}.
  There is an extensive literature on the Kuramoto model, where one has studied equilibria and phenomena like phase-locking and synchronization, as well as extensions to kinetic models like the Kuramoto--Sakaguchi equation \cite{morales2023trend}; we will merely refer to the aforementioned recent work and the references therein.
  It was observed in \cite{hemmen1993lyapunov} that \eqref{eq:kuramoto} has a gradient flow structure; however, the corresponding energy functional is not convex.
  Applying this idea to the reduction of the Cucker--Smale system, it can be seen as a gradient flow for the scalar potential
  \begin{equation}\label{eq:pot:part}
    \frac12 \sum_{j=1}^{N}\sum_{k=1}^{N}m_j m_k \int_0^{x_j-x_k} \Phi(y)\dee y -\sum_{j=1}^{N}m_j \bar{\psi}_j x_j,
  \end{equation}
  and the study of clustering in \cite{ha2018first,ha2019complete,zhang2020complete} shows that the corresponding trajectories are well-defined between collisions.
  In their analysis, the trajectories are allowed to cross, and the natural velocities \(\{\bar{\psi}_i\}_{i=1}^N\) provide an ordering of the asymptotic trajectories.
  However, Leslie and Tan \cite{leslie2024finite} argue that this crossing of trajectories is not well-suited for hydrodynamic limits of the particle system, and so they advocate for a different point of view, which we describe next.
  
  \subsubsection{An associated scalar balance law}
  Returning to the continuum setting, we introduce the quantity
  \begin{equation}\label{eq:psi}
    \psi_t \coloneqq v_t + \Phi\ast\rho_t,
  \end{equation}
  which can be seen as the continuum version of the discrete conserved variable from the previous section.
  Then one can formally rewrite \eqref{eq:EA} as
  \begin{subequations}\label{eq:EAalt}
    \begin{align}
      \del_t \rho_t + \del_x(\rho_t v_t) &= 0, \label{eq:EAalt:den} \\
      \del_t(\rho_t \psi_t) + \del_x(\rho_t \psi_t v_t) &= 0. \label{eq:EAalt:mom}
    \end{align}
  \end{subequations}
  Based on the above system and drawing inspiration from \cite{brenier1998sticky}, Leslie and Tan \cite{leslie2023sticky} derive an associated scalar balance law for \(M_t = \rho_t((-\infty,x])\),
  \begin{equation}\label{eq:blaw}
    \del_t M_t + \del_x A(M_t) = (\phi\ast M_t)\del_xM _t,
  \end{equation}
  where the flux function \(A\) depends only on the initial data.
  The initial density \(\bar{\rho}\) is assumed to be a compactly supported probability measure, i.e., \(\bar{\rho} \in \Pscr_{\mathrm{c}}(\R)\), while the initial velocity \(\bar{v}\) is essentially bounded with respect to \(\bar{\rho}\), i.e., \(\bar{v} \in \Lp{\infty}(\R,\bar{\rho})\).
  After establishing existence and uniqueness of entropy solutions of \eqref{eq:blaw} in the sense of Kru\v{z}kov, they show, using the \(BV\)-calculus of Vol'pert \cite{volpert1967spaces}, that these solutions provide distributional solutions of \eqref{eq:EA} exhibiting ``sticky dynamics''.
  Furthermore, with this framework they study clustering for the Euler-alignment system \eqref{eq:EA} in \cite{leslie2024finite}, and show that the global clustering behavior can be deduced from the initial data through the flux function \(A\).
  
  \subsection{The aim of this paper}
  Before stating the aim and main results of this paper, we present the motivation for our study.
  
  \subsubsection{Gradient flows and entropy solutions for pressureless Euler systems}
  Pressureless Euler systems similar to \eqref{eq:EA} have previously been studied using gradient flows in Hilbert spaces.
  Suppose \(\rho_t\) is a probability measure for \(t\ge 0\), denoted \(\rho_t \in \Pscr(\R)\), and that \(v_t\) is such that the momentum  \(v_t \rho_t\) is a Radon measure, denoted \(v_t \rho_t \in \Mscr(\R)\).
  In particular, we are interested in probability measures \(\rho_t\) with finite second moment and velocities \(v_t\) which are square-integrable with respect to \(\rho_t\).
  That is, \(\rho_t \in \Pscr_2(\R)\) and \(v_t \in \Lp{2}(\R,\rho_t)\), meaning
  \begin{equation*}
    \int_\R \abs{x}^2 \dee\rho_t < \infty \quad \text{and} \quad \int_\R \abs{v_t}^2 \dee\rho_t < \infty.
  \end{equation*}
  Then it turns out that \(\Pscr_2(\R)\) is isometric to the convex cone \(\Kscr \subset L^2(0,1)\) of nondecreasing functions.
  That is, any \(\rho_t \in \Pscr_2(\R)\) can be associated with an \(X_t \in \Kscr\), and vice versa.
  
  This framework was used by Natile and Savar\'{e} to study the pressureless Euler system \cite{natile2009wasserstein}, corresponding to \eqref{eq:EA} with \(\phi \equiv 0\).
  The idea is to use a similar reduction to a first-order system to see the initial (natural) velocity as the gradient of a linear functional.
  To ensure that solutions remain in \(\Kscr\), this functional is augmented with the indicator function \(I_{\Kscr}\), which in turn replaces the gradient with a set-valued subdifferential.
  The convexity of the functional makes the subdifferential a maximally monotone operator. Hence, by the theory of Br\'{e}zis \cite{brezis1973operateurs}, there is a unique solution \(X_t\), with velocity \(V_t\), evolving according to the minimal element of the subdifferential.
  This gives rise to a distributional solution of the original system which turns out to be globally sticky, that is, mass that clusters will remain clustered.
  In particular, they recover results of \cite{brenier1998sticky} obtained with a conservation law-approach.
  
  In what can be seen as an extension of the above framework, called the Lagrangian evolution, Brenier et al.\ \cite{brenier2013sticky} study the pressureless Euler system with a self-interacting force field \(f[\rho_t]\) driving the momentum equation.
  Under certain assumptions on \(f\), the resulting force  can be seen as a Lipschitz perturbation of the subdifferential in \cite{natile2009wasserstein}, which is still covered by the results of Br\'{e}zis, i.e., \cite[Theoreme 3.17, Remarque 3.14]{brezis1973operateurs}.
  A specific case covered by their theory is the Euler--Poisson system, which is also shown to have globally sticky solutions in the case of an attractive force.
  Note that this form of differential inclusion can be framed as an evolution variational inequality, and both \cite{natile2009wasserstein} and \cite{brenier2013sticky} make use of additional results from \cite{savare1996weak}, see also \cite{savare1993approximation}, to have better estimates for the velocity \(V_t\).
  
  On a slightly different note, \cite{bonaschi2015equivalence} concerns a singular, nonlocal interaction equation, which can be regarded as a first-order version of the Euler--Poisson system.
  For this equation one can also formulate \(\Lp{2}\)-gradient flow solutions, as well as, in the spirit of \cite{brenier1998sticky}, entropy solutions of an associated conservation law.
  The authors prove that these notions of solution are in fact equivalent, by means of passing to a limit in particle approximations.
  Based on an observation relating the minimal evolution of the gradient flow and the Ole\u{\i}nik E condition for conservation laws, \cite{carrillo2023equiv} extended the equivalence from \cite{bonaschi2015equivalence} to pressureless Euler systems with appropriate forcing terms \(f\).
  Specifically, one finds that the solutions of Euler--Poisson obtained in \cite{brenier2013sticky} using gradient flows and solutions obtained from entropy solutions of a conservation law in \cite{nguyen2015one} are equivalent.
  We also mention the work \cite{ben2024mean}, where \cite{bonaschi2015equivalence} and \cite{leslie2023sticky} inspired the use of a balance law to study the mean-field limit of a second-order particle system for opinion dynamics.
  
  \subsubsection{Aim and main results}
  We mentioned before how sticky solutions of \eqref{eq:EA} can be obtained from entropy solutions of the balance law \eqref{eq:blaw} for initial data \((\bar{\rho},\bar{v}) \in \Pscr_\mathrm{c}(\R)\times\Lp{\infty}(\R,\bar{\rho})\).
  Considering the aforementioned equivalences between gradient flows and entropy solutions, it is then tempting to follow \cite{natile2009wasserstein,brenier2013sticky} in extending to \((\bar{\rho},\bar{v}) \in \Pscr_2(\R)\times\Lp{2}(\R,\bar{\rho})\)
  by formulating an associated gradient flow solution.
  However, there are some apparent obstacles.
  For one, \eqref{eq:EA} is not covered directly by the theory in \cite{brenier2013sticky}, as the forcing term is of the form \(f[\rho_t,v_t]\) rather than \(f[\rho_t]\).
  Moreover, in \cite{leslie2023sticky} they allow for weakly singular \(\phi\), for which the corresponding alignment force cannot be seen as a Lipschitz perturbation of the subdifferential in \cite{natile2009wasserstein}.
  
  We take a different approach to overcome this problem, regarding instead the alignment force as part of the subdifferential.
  Indeed, unlike for the Kuramoto model, the functional \eqref{eq:pot:part} is in fact convex and its \(\Lp{2}\)-version can then be augmented with the indicator function \(I_{\Kscr}\) to form a convex functional.
  
  Yet another indication that the framework of \(\Lp{2}\)-gradient flows is suitable for the problem at hand is found in \cite{leslie2024finite}.
  There the authors use the flux function \(A\) from \eqref{eq:blaw}, which can be regarded as the primitive of the \(\Lp{2}\)-version of the natural velocities, and its lower convex envelope \(A^{**}\) to predict the clustering behavior of the sticky solutions.
  The suggestion comes from a relation in \cite{natile2009wasserstein} which connects the projection of a function onto the convex cone \(\Kscr\) and the lower convex envelope of its primitive.
  In this sense, \(A^{**}\) encodes the projection of the natural velocity onto this cone.
  
  With this in mind, we seek the following type of solutions to the Euler-alignment system.
  \begin{dfn}[Distributional solutions of the Euler-alignment system]\label{dfn:EA:distsol}
    The pair \((\rho_t, v_t) \in \Pscr_2(\R) \times \Lp{2}(\R,\rho_t)\) is a distributional solution of the initial value problem for the Euler-alignment system if it satisfies \eqref{eq:EA} in the distributional sense, and for initial values \((\bar{\rho},\bar{v})\) we have
    \begin{equation}\label{eq:EA:init}
      \lim\limits_{t\to0+}\rho_t = \bar{\rho} \quad \text{in} \ \Pscr_2(\R) \qquad \text{and} \qquad \lim\limits_{t\to0+} v_t \rho_t = \bar{v} \bar{\rho} \quad \text{in} \ \Mscr(\R).
    \end{equation}
  \end{dfn}
  
  The paper is organized as follows.
  Section \ref{s:particles} provides a motivation of the gradient flow solution for the particle case.
  Some auxiliary results on optimal transport and convex analysis are presented in Section \ref{s:aux}, leading to the associated Lagrangian solutions in Section \ref{s:Lag}.
  Section \ref{s:other} details how the Lagrangian solution yields a solution of the original Euler-alignment system, as well as for the associated balance law.
  Finally, in Section \ref{s:cluster} we show how the Lagrangian solutions can be used to derive the clustering properties of the sticky solutions.
  
  We finish this section with a colloquial presentation of our main results, while we point to later theorems and corollaries for precise statements and full-fledged details.
	\begin{res}[Lagrangian solution]\label{res:1}
		There is a unique Lagrangian, or \(\Lp{2}\)-gradient flow, solution concept associated with the Euler-alignment system \eqref{eq:EA}.
	\end{res}
	That is, given the initial value problem for \eqref{eq:EA}, we can define an associated convex potential, or functional, which includes the term \(I_{\Kscr}\), the indicator function of the convex cone \(\Kscr\) of nondecreasing \(\Lp{2}(0,1)\)-functions.
	Because of this term, we must relax the notion of a gradient and instead consider the set-valued subdifferential of the functional, where an element is often called a subgradient.
	Then there is a unique map \(t \mapsto X_t \in \Kscr \) which evolves according to the subgradient of minimal \(\Lp{2}\)-norm.
	This is the Lagrangian solution, and it can be seen as a generalization of the particle positions for the Cucker--Smale-type system \eqref{eq:CS}, which we recall had a gradient-flow structure based on the potential \eqref{eq:pot:part}.
	The details of Result \ref{res:1} can be found in Theorem \ref{thm:LagSol}.
	
	\begin{res}[Particle approximation]\label{res:2}
		The Lagrangian solution can be realized as the limit of solutions of sticky-particle Cucker--Smale dynamics. 
		From these approximations, we deduce that the solution is globally sticky and features projection formulas and a semigroup property. 
	\end{res}
	The solutions of the sticky-particle Cucker--Smale dynamics evolve according to \eqref{eq:CS}, with the additional rule that particles which collide stick together in such a way that momentum is conserved.
	For such a particle solution we can construct a step function on the interval \([0,1)\), where the length of the \(i\)\textsuperscript{th} step is the mass of the \(i\)\textsuperscript{th} particle, while the value at that step is the position of said particle; this turns out to be a Lagrangian solution.
	We can then approximate the Lagrangian solution \(X_t\) by first approximating its initial data and velocity with step functions, and then let these evolve according to the corresponding sticky-particle dynamics.
	The convergence then follows from stability properties of Lagrangian solutions and convergence of the initial approximations, see Theorem \ref{thm:convergence} for details.
	
	Since the particle solutions are globally sticky, we can prove that also the limiting Lagrangian solution has this property, meaning that concentrated mass remains concentrated for all time.
	This allows us to express the Lagrangian solutions in terms of projection formulas related to the cone \(\Kscr\). These projections provide a semigroup property for the solutions, and may be of interest for numerical implementations.
	Further details on the stickiness property and the projection formulas are respectively provided in Corollary \ref{cor:stickyness} and Proposition \ref{prp:sticky:projection}.
	
	\begin{res}[Distributional solution]\label{res:3}
		The globally sticky Lagrangian solution gives rise to a distributional solution of the Euler-alignment system \eqref{eq:EA} and an entropy solution of the balance law \eqref{eq:blaw}, both of which are uniquely defined.
	\end{res}
	Let \(X_t\) be the Lagrangian solution associated with a set of initial data for \eqref{eq:EA}, and define \(\rho_t\) as the pushforward of the Lebesgue measure on \((0,1)\) through the map \(X_t\).
	Then there is a unique \(v_t \in \Lp{2}(\R,\rho_t)\) for which \(v_t \circ X_t\) coincides with the velocity of \(X_t\), and we can show that the pair \((\rho_t, v_t)\) satisfies \eqref{eq:EA} in the distributional sense.
	On the other hand, introducing \(M_t\) as the generalized inverse of \(X_t\), we can show that this in fact satisfies \eqref{eq:blaw} in the weak sense.
	This relaxes the assumptions on initial data for \eqref{eq:EA} from \cite{leslie2023sticky}, since \(\Pscr_{\mathrm{c}}(\R) \subset \Pscr_2(\R)\) and \(\Lp{\infty}(\R,\rho) \subset \Lp{2}(\R,\rho)\) for \(\rho \in \Pscr_{\mathrm{c}}(\R)\).
	We refer to Theorems \ref{thm:EAsol} and \ref{thm:blawsol} for the specifics of Result \ref{res:3}.
	
	\begin{res}[Clustering]\label{res:4}
		We can deduce the clustering behavior from the Lagrangian solution.
		In particular, clusters are identified from monotonicity properties of the associated natural velocity.
	\end{res}
	In analogy with the natural velocities \(\{\bar{\psi}_i\}_i\) for the Cucker--Smale-type system \eqref{eq:CS}, there is also a natural velocity \(\bar{\Psi} \in \Lp{2}(0,1)\) for the Lagrangian solution.
	For instance, it turns out that in the same way that particles will collide whenever \(\bar{\psi}_i\) are not in increasing order, mass will cluster in the Lagrangian solution if \(\bar{\Psi} \notin \Kscr\).
	To this end, denote by \(A\) the primitive of \(\bar{\Psi}\), this is in fact how the flux function in \eqref{eq:blaw} is defined, and let \(A^{**}\) be its lower convex envelope on \([0,1]\).
	By a result in \cite{natile2009wasserstein}, the right-hand derivative of \(A^{**}\) is the projection of \(\bar{\Psi}\) on \(\Kscr\).
	
	The image of the projection partitions the domain \((0,1)\) into sets of strictly increasing averaged natural velocity, and these sets will never cluster with one another.
	In order to study the formation of clusters, we further subdivide these sets by comparing \(A\) with \(A^{**}\).
	We can show that segments of \((0,1)\) where \(A\) deviates from \(A^{**}\) correspond to mass which clusters in finite time; on these segments \(\bar{\Psi}\) necessarily deviates from its projection on \(\Kscr\).
	On the other hand, mass will never cluster on segments where \(A\) coincides with \(A^{**}\) and \(\bar{\Psi}\) is strictly increasing.
	Finally, a segment where \(A\) coincides with \(A^{**}\) and both are linear is also a cluster, but whether the mass clusters in finite or infinite time depends on the regularity of the communication protocol \(\phi\).
	These three cases are respectively coined \textit{supercritical}, \textit{subcritical} and \textit{critical} in \cite{leslie2024finite}, and now we have seen the bridge between their study of the lower convex envelope \(A^{**}\) and the convex cone \(\Kscr\) containing the Lagrangian solutions.
	Theorem \ref{thm:clustering} provides the details of Result \ref{res:4}.
	On a final, related note, we can deduce whether different subsets or clusters flock, i.e., remain close, by considering the communication strength of \(\phi\) over large distances, i.e., its tail \eqref{eq:tail}.
	For a thin tail, whether two subsets drift apart or not is the result of a competition between the `size of the tail' \(\norm{\phi}_{\Lp{1}}\) and the difference in natural velocity for the subsets.
	If instead the tail is fat, there is always an upper bound on the distance between two subsets.
  
  We emphasize that the techniques we use are inherently one-dimensional in nature, and so these ideas are not directly applicable in higher dimensions.
  
  \section{Motivation in particle case}\label{s:particles}
  We will follow \cite{brenier2013sticky} in first motivating the \(\Lp{2}\)-gradient flow using the particle system.
  
  \subsection{Deriving the particle dynamics}\label{ss:partintro}
  We denote \(\bs{x} = (x_1,\dots,x_N) \in \R^N\), \(\bs{v} = (v_1,\dots,v_N) \in \R^N\) and \(\bs{m} = (m_1,\dots,m_N) \in \R^N_+\), where \(\R^N_+\) denotes the elements of \(\R^N\) with strictly positive entries.
  Let us introduce the auxiliary functions
  \begin{equation}\label{eq:psi:part}
    \psi_i(t) \coloneqq v_i(t) + \sum_{j=1}^{N}m_j \Phi\left(x_i(t)-x_j(t)\right), \qquad \bar{\psi}_i \coloneqq \bar{v}_i + \sum_{j=1}^{N}m_j \Phi\left(\bar{x}_i-\bar{x}_j\right), \quad i \in \{1,\dots,N\},
  \end{equation}
  which we collect in vectors \(\bs{\psi} = \bs{v} + \bs{\Phi}^\ast_{\bs{m}}(\bs{x})\) and \(\bar{\bs{\psi}} = \bar{\bs{v}} + \bs{\Phi}^\ast_{\bs{m}}(\bar{\bs{x}})\), where we have introduced \(\bs{\Phi}^\ast_{\bs{m}}\) as a shorthand notation for the convolution-like quantities in \eqref{eq:psi:part}.
  Then the particle dynamics \eqref{eq:EAparticles} can be rephrased as the following system of differential equations,
  \begin{equation}\label{eq:EA:part:DE}
    \dot{x}_i = \bar{\psi}_i - \sum_{j=1}^N m_j \Phi(x_i-x_j), \quad x_i(0) = \bar{x}_i, \quad i \in \{1,\dots,N\}.
  \end{equation}
  We do not want trajectories of this system to cross, and so we require the particles to retain their initial ordering.
  To this end we introduce the closed, convex cone
  \begin{equation}\label{eq:cone:part}
    \K^N \coloneqq \left\{ \bs{x} \in \R^N \colon x_1 \le x_2 \le \dots \le x_N \right\},
  \end{equation}
  and assume our initial data belongs to this cone, that is, \(\bar{\bs{x}} \in \K^N\).
  Note that for \(\bs{x} \in \K^N\) and \(\bs{m} \in \R_+^N\) we can define the \(\bs{m}\)-weighted Euclidean \(p\)-norm
  \begin{equation}\label{eq:euclid}
    \norm{\bs{x}}_{\bs{m},p} \coloneqq \left(\sum_{i=1}^{N}m_i \abs{x_i}^p\right)^{1/p}, \qquad \norm{\bs{x}}_{\bs{m},\infty} = \norm{\bs{x}}_{\infty}.
  \end{equation}
  For \(p=2\) we can define an associated \(\bs{m}\)-weighted inner product
  \begin{equation*}
    \langle \bs{y}, \bs{x} \rangle_{\bs{m}} \coloneqq \sum_{i=1}^{N}m_i y_i x_i,
  \end{equation*}
  and for convenience we write \(\norm{\cdot}_{\bs{m}} \coloneqq \norm{\cdot}_{\bs{m},2}\).
  \subsubsection{Short-time existence and uniqueness of solutions}
  We note that the right-hand side of \eqref{eq:EA:part:DE} is a continuous function in \(x_i\), and so existence of solutions \(\bs{x} \) in \(C^1([0,\infty);\R^N)\) is guaranteed.
  On the other hand, we cannot expect Lipschitz-continuity of the right-hand side, e.g., if \(\phi\) is weakly singular in the sense of \eqref{eq:w-sing}.
  However, owing to the monotonicity of \(\Phi^\ast_{\bs{m}}\) we can obtain uniqueness anyway, at least until particles meet, i.e., when \(\bs{x}\) hits the boundary \(\del\K^N\) given by
  \begin{equation}\label{eq:OX:part}
    \del \K^N = \{\bs{x} \in \K^N \colon \Omega_{\bs{x}} \neq \emptyset \}, \qquad \Omega_{\bs{x}} \coloneqq \{j \colon x_j = x_{j+1}, j \in \{1,\dots,N-1\}\}.
  \end{equation}
  Let us specify what is meant by monotonicity here. For \(\bs{x}, \bs{y} \in \R^N\) we compute
  \begin{align*}
    \ip{\bs{x}-\bs{y}}{\Phi^{\ast}_{\bs{m}}(\bs{x})-\Phi^{\ast}_{\bs{m}}(\bs{y})}_{\bs{m}} &= \sum_{i=1}^N m_i (x_i-y_i) \sum_{j=1}^N m_j \left[\Phi(x_i-x_j) -\Phi(y_i-y_j)\right] \\
    &= \frac12 \sum_{i=1}^N\sum_{j=1}^N m_i m_j [(x_i-x_j)-(y_i-y_j)]\left[\Phi(x_i-x_j) -\Phi(y_i-y_j)\right] \ge 0,
  \end{align*}
  where the second equality follows from \(\Phi\) being odd, and the final inequality from \(\Phi\) being nondecreasing.
  Hence, this monotonicity yields a one-sided Lipschitz condition, and we can show a stability estimate.
  For \(i \in \{1,2\}\), let \(\bs{x}_i(t)\) be a solution of \eqref{eq:EA:part:DE} with initial data \((\bs{x}_i(0), \bs{v}_i(0))=(\bar{\bs{x}}_i, \bar{\bs{v}}_i)\) so that one has \(\bar{\bs{\psi}}_i = \bar{\bs{v}}_i + \bs{\Phi}^\ast_{\bs{m}}(\bs{\bar{x}}_i)\).
  Then we have
  \begin{equation*}
    \diff{}{t} \frac12 \norm{\bs{x}_1-\bs{x}_2}_{\bs{m}}^2 = \ip{\bs{x}_1-\bs{x}_2}{\dot{\bs{x}}_1-\dot{\bs{x}}_2}_{\bs{m}} \le \ip{\bs{x}_1-\bs{x}_2}{\bar{\bs{\psi}}_1-\bar{\bs{\psi}}_2}_{\bs{m}} \le \norm{\bs{x}_1-\bs{x}_2}_{\bs{m}} \norm{\bar{\bs{\psi}}_1-\bar{\bs{\psi}}_2}_{\bs{m}},
  \end{equation*}
  which implies
  \begin{equation*}
    \diff{}{t} \norm{\bs{x}_1-\bs{x}_2}_{\bs{m}} \le \norm{\bar{\bs{\psi}}_1-\bar{\bs{\psi}}_2}_{\bs{m}},
  \end{equation*}
  and in turn
  \begin{equation*}
    \norm{\bs{x}_1(t)-\bs{x}_2(t)}_{\bs{m}} \le \norm{\bs{x}_1(s)-\bs{x}_2(s)}_{\bs{m}} + (t-s)\norm{\bar{\bs{\psi}}_1-\bar{\bs{\psi}}_2}_{\bs{m}}.
  \end{equation*}
  Suppose we start with distinct initial particles, i.e., \(\bar{x}_1 < \bar{x}_2 < \dots < \bar{x}_N\), meaning \(\bar{\bs{x}} \in \mathrm{int}(\K^N)\), the interior of \(\K^N\).
  Then we have a unique solution \(\bs{x}(t) \in \K^N\) at least for a short time, until \(\bs{x}(t) \in \del \K^N \).
  To ensure that the solution remains in \(\K^N\) after particles meet, we will in the following provide a well-defined procedure for resolving the dynamics in this case.
  
  \subsubsection{Collision dynamics and the differential inclusion}
  For \(\bs{x} \in \K^N\), the \textit{tangent cone} to \(\K^N\) at \(\bs{x}\), cf.\ \cite[Definition 6.38]{bauschke2017convex}, is defined as the following closure,
  \begin{equation}\label{eq:TC:part}
    T_{\bs{x}}\K^N \coloneqq \mathrm{cl}\{\vartheta(\bs{y}-\bs{x}) \colon \bs{y} \in \K^N, \vartheta \ge 0 \}.
  \end{equation}
  This is the cone in which the velocity of \(\bs{x}\) should belong for \(\bs{x}\) to remain within \(\K^N\), and in our setting \eqref{eq:TC:part} is equivalent to
  \begin{equation*}
    T_{\bs{x}}\K^N = \{\bs{v} \in \R^N \colon v_j \le v_{j+1} \ \text{for all} \ j \in \Omega_{\bs{x}}\},
  \end{equation*}
  where we recall \(\Omega_{\bs{x}}\) from \eqref{eq:OX:part}.
  In particular, we see that if \(\Omega_{\bs{x}} = \emptyset\), i.e., \(\bs{x} \in \mathrm{int}(\K^N)\), then \(T_{\bs{x}}\K^N = \R^N\).
  Assume a subset of particles \(\Jcal_i(t) \subset \{1,\dots,N\}\) collides at time \(t\), where we define
  \begin{equation*}
    \Jcal_i(t) \coloneqq \left\{j \colon x_j(t) = x_i(t)\right\}, \qquad i_*(t) = \min\Jcal_i(t), \qquad i^*(t) \coloneqq \max \Jcal_i(t).
  \end{equation*}
  Furthermore we assume the completely inelastic collision rule
  \begin{equation}\label{eq:v:proj}
    \bs{v}(t+) = \Proj{T_{\bs{x}(t)}\K^N}\bs{v}(t-)
  \end{equation}
  leading to
  \begin{equation*}
    v_i(t+) = \frac{\sum_{j\in\Jcal_i(t)}m_j v_j(t-)}{\sum_{j\in\Jcal_i(t)}m_j},
  \end{equation*}
  which is natural in the sense that it conserves momentum.
  Now, as noted in \cite{leslie2023sticky}, by continuity of the trajectory \(\bs{x}(t)\) and the continuity of \(\Phi\) it follows that \(\psi_j(t-)-v_j(t-) = \psi_i(t-)-v_i(t-)\) for all \(j \in \Jcal_i(t)\).
  This in turn leads to
  \begin{equation}\label{eq:RH:part}
    \psi_j(t+) = \frac{\sum_{j\in\Jcal_i(t)}m_j \psi_j(t-)}{\sum_{j\in\Jcal_i(t)}m_j} = \frac{\sum_{j\in\Jcal_i(t)}m_j \bar{\psi}_j}{\sum_{j\in\Jcal_i(t)}m_j}
  \end{equation}
  where the final identity comes from \(\psi_i(t)\) being constant between collisions by virtue of \eqref{eq:EA:part:DE}.
  For convenience we take the velocities \(v_i(t)\), and then consequently also \(\psi_i(t)\), to be right-continuous functions of \(t\).
  Afterwards and until the next collision, the clustered particle will only be affected by particles not belonging to the cluster.
  Indeed, for \(k \in \Jcal_i(t)\) we have
  \begin{equation*}
    \dot{v}_k(t) = -\sum_{\substack{j=1 \\ x_j(t)\notin \Jcal_i(t)}}^N m_j \phi(x_k(t)-x_j(t))(v_k(t)-v_j(t))
  \end{equation*}
  or, since \(\Phi(0) = 0\),
  \begin{equation*}
    v_k(t) + \sum_{j=1}^n m_j \Phi(x_k(t)-x_j(t)) = \frac{\sum_{j\in\Jcal_i(t)}m_j \bar{\psi}_j}{\sum_{j\in\Jcal_i(t)}m_j}.
  \end{equation*}
  Moreover, amassed particles will not break apart again, since every particle in labeled by \(\Jcal_i(t)\) moves with the same velocity.
  As pointed out in \cite{brenier2013sticky}, this sticky evolution can equivalently be defined by relabeling the amassed particles as a single new particle with mass given by the sum of the previous masses.
  At every collision, the number of particles in the system would then decrease from the initial \(N\).
  Recalling \eqref{eq:RH:part}, the collision dynamics can therefore be seen as a projection of the initial vector \(\bar{\bs{\psi}}\) of natural velocities onto the tangent cone of \(\bs{x}(t)\).
  That is,
  \begin{equation}\label{eq:proj:psi}
    \bs{\psi}(t+) = \Proj{T_{\bs{x}(t)}\K^N}\bs{\psi}(t-) = \Proj{T_{\bs{x}(t)}\K^N}\bar{\bs{\psi}}.
  \end{equation}
  From this we see how the collision, or clustering, dynamics only depends on \(\bar{\bs{\psi}}\), as expected from the particle dynamics in \cite{leslie2024finite}.
  From the above arguments, we have the following result.
  \begin{lem}\label{lem:EA:part:unique}
    Let \((\bar{\bs{x}}, \bar{\bs{v}}) \in \mathrm{int(\K^N)\times\R^N}\) and define \(\bar{\bs{\psi}} \in \R^N\) as in \eqref{eq:psi:part}.
    Then, for \(t \ge 0\) there is a uniquely defined, globally sticky solution \((\bs{x}(t),\bs{v}(t)) \in \K^N \times T_{\bs{x}(t)}\K^N\) of \eqref{eq:EA:part:DE} satisfying \((\bs{x}(0),\bs{v}(0)) = (\bar{\bs{x}},\bar{\bs{v}})\).
  \end{lem}
  
  Following \cite[Section 1.2]{brenier2013sticky}, we then argue that the instantaneous force that changes the velocity on impact at \(\bs{x} \in \del\K^N\) should belong to the \textit{normal cone}\footnote{A cone very similar to \(\K^N\), and its polar cone, is treated in \cite[Exercise 6.16]{bauschke2017convex}.} \(N_{\bs{x}}\K^N\), defined as
  \begin{equation}\label{eq:NC:part}
  	N_{\bs{x}}\K^N \coloneqq \left\{ \bs{n} \in \R^N \colon \ip{\bs{n}}{\bs{y}-\bs{x}}_{\bs{m}} \le 0 \ \text{for all} \ \bs{y} \in \K^N \right\}.
  \end{equation}
  Then we can incorporate the collision dynamics in \eqref{eq:EAparticles} by rephrasing it as the second-order differential inclusion
  \begin{equation}
    \dot{x}_i = v_i, \qquad \dot{v}_i + N_{\bs{x}}\K^N \ni -\sum_{\substack{j=1 \\ x_j\neq x_i}}^{N} m_j \phi(x_i-x_j)(v_i-v_j).
  \end{equation}
  or equivalently, using \eqref{eq:psi:part},
  \begin{equation}\label{eq:EA:part:DI}
    \dot{\bs{x}} = \bs{v}, \qquad \dot{\bs{\psi}} + N_{\bs{x}}\K^N \ni \bs{0}.
  \end{equation}
  As observed in \cite{natile2009wasserstein}, if \(\bs{x}\colon [0,\infty) \to \K^N\) satisfies a global stickiness condition, there is a monotonicity property for the normal cones, namely
  \begin{equation*}
    N_{\bs{x}(s)}\K^N \subset N_{\bs{x}(t)}\K^N \quad \text{for all } s < t.
  \end{equation*}
  Then, for \(\bs{\xi} \colon [0,\infty) \to \R^N\) satisfying \(\bs{\xi}(t) \in N_{\bs{x}(t)}\K^N\), e.g., \(-\dot{\bs{\psi}}(t)\) in \eqref{eq:EA:part:DI}, we obtain
  \begin{equation*}
    \int_s^t \bs{\xi}(r)\dee r \in N_{\bs{x}(t)}\K^N \quad \text{for all } s < t.
  \end{equation*}
  Combining the above with a formal integration of the second equation of \eqref{eq:EA:part:DI} then yields 
  \begin{equation*}
    \bs{\psi}(t) + N_{\bs{x}(t)}\K^N \ni \bar{\bs{\psi}},
  \end{equation*}
  which we in turn rephrase as the first-order differential inclusion
  \begin{equation}\label{eq:EA:DI:part}
    \dot{\bs{x}} + \bs{\Phi}^\ast_{\bs{m}}(\bs{x}) + N_{\bs{x}}\K^N \ni \bar{\bs{\psi}}.
  \end{equation}
  This relation lays the foundation for our study of the problem in the continuum case.
  
  \subsubsection{The barycentric lemma}
  Consider again a subset of particles \(\Jcal_i(t)\) colliding at time \(t\); their velocities must necessarily satisfy
  \begin{equation*}
    v_{i_*(t)}(t-) \ge v_{i_*(t)+1}(t-) \ge \dots \ge v_{i^*(t)-1}(t-) \ge v_{i^*(t)}(t-), 
  \end{equation*}
  or else they would not have collided in the first place.
  However, by the continuity of the trajectories \(x_i\) and the continuity of \(\Phi\) we also have
  \begin{equation*}
    \psi_{i_*(t)}(t-) \ge \psi_{i_*(t)+1}(t-) \ge \dots \ge \psi_{i^*(t)-1}(t-) \ge \psi_{i^*(t)}(t-).
  \end{equation*}
  Combining the above chain of inequalities with the collision dynamics \eqref{eq:v:proj}, we obtain
  \begin{equation}\label{eq:barycentric}
    \frac{\sum_{j=k}^{i^*(t)}m_j \psi_j(t-)}{\sum_{j=k}^{i^*(t)}m_j} \le \psi_i(t+) = \frac{\sum_{j\in\Jcal_i(t)}m_j \psi_j(t-)}{\sum_{j\in\Jcal_i(t)}m_j} \le \frac{\sum_{j=i_*(t)}^k m_j \psi_j(t-)}{\sum_{j=i_*(t)}^k m_j},
  \end{equation}
  and, relying once more on \(\psi_i(t)\) being constant between collisions, we can write this as
  \begin{equation}\label{eq:Oleinik:part}
    \frac{\sum_{j=k}^{i^*(t)}m_j \bar{\psi}_j}{\sum_{j=k}^{i^*(t)}m_j} \le \psi_i(t+)  = \frac{\sum_{j\in\Jcal_i(t)}m_j \bar{\psi}_j}{\sum_{j\in\Jcal_i(t)}m_j} \le \frac{\sum_{j=i_*(t)}^k m_j \bar{\psi}_j}{\sum_{j=i_*(t)}^k m_j}.
  \end{equation}
  This is the \textit{barycentric lemma} used in \cite{brenier1998sticky} for the case \(\phi \equiv 0\), and which was shown in \cite{leslie2023sticky} to still hold in our current case.
  It turns out that the barycentric lemma is a particular, in fact, particle case of the Ole\u{\i}nik E condition \cite{oleinik1959uniqueness} for the flux function \(A\) from the balance law \eqref{eq:blaw}.
  We return to this matter in Section \ref{s:other}.
  
  Figure \ref{fig:discflux} illustrates the barycentric lemma for six particles with different masses \(m_i\) and natural velocities \(\bar{\psi}_i\), \(i \in \{1,\dots,6\}\). The natural velocity \(\bar{\psi}_i\) of the \(i\)\textsuperscript{th} particle is the slope of the \(i\)\textsuperscript{th} solid line segment,  and the segments can be seen as a piecewise linear interpolation of a (flux) function \(A\) with breakpoints at \(\theta_i = \sum_{j=1}^{i}m_j\).
  If particles 3, 4 and 5 were to collide, the weighted average of their natural velocities is the slope of the dash-dotted line segment, which is the lower convex envelope of the interpolating function on the interval \([\theta_2,\theta_5]\). Observe that \eqref{eq:Oleinik:part} holds for this case.
  \begin{figure}
    \includegraphics[width=0.7\linewidth]{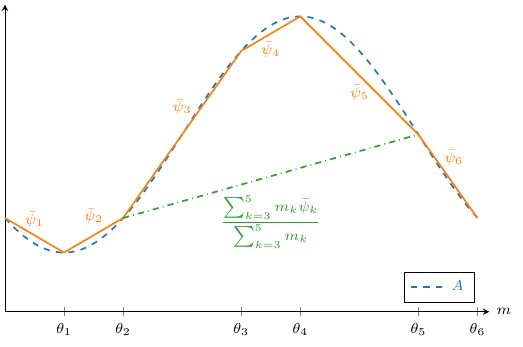}
    \caption{Illustration of natural velocities \(\bar{\psi}_i\) for six particles with different masses \(m_i\), \(i \in \{1,\dots,6\}\), where we have introduced the breakpoints \(\theta_i = \sum_{k=1}^{i}m_k\).}
    \label{fig:discflux}
  \end{figure}
  
  \section{Auxiliary results}\label{s:aux}
  In order to define Lagrangian, or \(\Lp{2}\)-gradient flow, solutions for the Euler-alignment system, we will need some auxiliary results which play a central role in \cite{natile2009wasserstein,brenier2013sticky}.
  
  \subsection{Some results from optimal transport}
  Consider a probability measure \(\rho \in \Pscr(\R)\), for which we define its right-continuous cumulative distribution function
  \begin{equation}\label{eq:M}
    M_\rho(x) = \rho((-\infty,x]), \quad x \in \R.
  \end{equation}
  Then \(\del_x M_\rho = \rho\) in \(\Dscr'(\R)\), i.e., in the distributional sense.
  We can then define its right-continuous monotone rearrangement, or generalized inverse, through
  \begin{equation}\label{eq:X}
    X_\rho \coloneqq \inf\left\{ x \colon M_\rho(x) > m \right\}, \quad m \in \Omega,
  \end{equation}
  where \(\Omega \coloneqq (0,1)\).
  We denote the restriction of the one-dimensional Lebesgue measure \(\mathcal{L}^1\) to \(\Omega\) by \(\mathfrak{m}\coloneqq \mathcal{L}^1|_\Omega\), such that we have the push-forward relations
  \begin{equation}\label{eq:push}
    X_{\rho\#}\mathfrak{m} = \rho, \qquad \int_\R \phiv(x) \dee\rho(x) = \int_\Omega \phiv(X_\rho(m)) \dee m
  \end{equation}
  for any Borel function \(\varphi \colon \R \to [0,\infty]\).
  
  The \(p\)-Wasserstein distance, or Kantorovich--Rubinstein metric, between two measures is defined as 
  \begin{equation}\label{eq:Wp}
    W_p^p(\rho_1,\rho_2) \coloneqq \min\left\{\int_{\R\times\R} \abs{x-y}^p \dee\varrho(x,y) \colon \varrho \in \Pscr(\R\times\R), \varpi^{i}_{\#}\varrho = \rho_i \right\},
  \end{equation}
  where \(\varpi^{i}\) is the projection on the \(i\)\textsuperscript{th} coordinate, i.e., \(\varpi^{i}(x_1,x_2) = x_i\).
  In one dimension, the unique optimal coupling of measures can be explicitly found using the monotone rearrangement. Indeed, by the Hoeffding--Frech\'{e}t theorem \cite[Theorem 2.18]{villani2003topics} the optimal coupling is given by
  \begin{equation}\label{eq:optcoupling}
    \varrho = X_{\rho_1,\rho_2\#}\mathfrak{m}, \qquad X_{\rho_1,\rho_2} = (X_{\rho_1}, X_{\rho_2}).
  \end{equation}
  A direct consequence of this is the identity
  \begin{equation*}
    W_p(\rho_1,\rho_2) = \left(\int_{\Omega} \abs{X_{\rho_1}-X_{\rho_2}}^p\dee \omega\right)^{1/p} = \norm{X_{\rho_1}-X_{\rho_2}}_{\Lp{p}(\Omega)}.
  \end{equation*}
  That is, the \(p\)-Wasserstein distance of two probability measures equals the \(\Lp{p}\)-distance of their monotone rearrangements.
  Let us then for \(p \in [1,\infty)\) introduce the space
  \begin{equation}\label{eq:Tp}
    \Tscr_p \coloneqq \left\{ (\rho,v) \colon \rho \in \Pscr_p(\R), v \in \Lp{p}(\R,\rho) \right\},
  \end{equation}
  for which we can define a semidistance \(U_p\) as follows,
  \begin{equation}\label{eq:Up}
    U_p^p((\rho_1,v_1),(\rho_2,v_2)) \coloneqq \int_{\R\times\R}\abs{v_1(x)-v_2(y)}^p \dee\varrho(x,y) = \norm{v_1\circ X_1 - v_2\circ X_2}_{\Lp{p}(\Omega)}^p,
  \end{equation}
  where the final identity again follows from \eqref{eq:optcoupling}.
  In turn, we can define a metric \(D_p\) through
  \begin{equation}\label{eq:Dp}
    D_p^p((\rho_1,v_1),(\rho_2,v_2)) \coloneqq W_p^p(\rho_1,\rho_2) + U_p^p((\rho_1,v_1),(\rho_2,v_2)),
  \end{equation}
  so that \((\Tscr_p, D_p)\) is a metric, but not complete, space, see \cite[Proposition 2.1]{natile2009wasserstein}.
  
  \subsection{Some results from convex analysis}
  Let \(f \colon \Xscr \to \R \cup \{+\infty\}\) be an extended real-valued function for some set \(\Xscr\).
  Its \textit{effective domain}, or just \textit{domain}, \(\dom(f)\) is the set of points where \(f\) is finite, i.e.,
  \begin{equation*}
    \dom(f) \coloneqq \{x \in \Xscr \colon f(x) < +\infty\}.
  \end{equation*}
  If \(\dom(f) \neq \emptyset\), then \(f\) is called \textit{proper}.
  Throughout the paper we will consider the real-valued Hilbert space \(\Lp{2}(\Omega)\), where we write \(\ip{\cdot}{\cdot} = \ip{\cdot}{\cdot}_{\Lp{2}(\Omega)}\) for its associated inner product and, sometimes for brevity, \(\norm{\cdot} = \norm{\cdot}_{\Lp{2}(\Omega)}\) for its norm.
  
  The metric projection onto a nonempty, closed, convex subset \(\Cscr\) of a Hilbert space is a well-defined Lipschitz map \(\Proj{\Cscr} \colon \Lp{2}(\Omega) \to \Cscr\), and for all \(X \in \Lp{2}(\Omega)\) it is characterized by
  \begin{equation}\label{eq:proj:C}
    Y = \Proj{\Cscr}X \iff Y \in \Cscr, \quad \ip{X-Y}{Z-Y} \le 0 \ \text{for all} \ Z \in \Cscr,
  \end{equation}
  see, e.g., \cite[Theorem 3.16]{bauschke2017convex} or \cite[Proposition 1.37]{peypouquet2015convex}.
  Denoting by \(\Cscr^\circ\) the element of minimal norm in a closed and convex set \(\Cscr \subset \Lp{2}(\Omega)\), \eqref{eq:proj:C} is equivalent to
  \begin{equation*}
    \norm{X-Y}_{\Lp{2}(\Omega)} = \min_{Z \in \Cscr} \norm{X-Z}_{\Lp{2}(\Omega)} \iff Y = \left(X - \Cscr\right)^\circ.
  \end{equation*}
  Let us now consider what can be seen as a generalization of \(\K^N\) from \eqref{eq:cone:part}, namely
  \begin{equation}\label{eq:K}
    \Kscr \coloneqq \{ X \in \Lp{2}(\Omega) \colon X \text{ is nondecreasing and right-continuous} \},
  \end{equation}
  which is a closed, convex cone in \(\Lp{2}(\Omega)\).
  The indicator function \(I_\Kscr\) of \(\Kscr\) is then given by
  \begin{equation}\label{eq:indicator}
    I_\Kscr(X) = \begin{cases}
      0, & X \in \Kscr, \\
      +\infty, & X \notin \Kscr.
    \end{cases}
  \end{equation}
  Since the set \(\Kscr\) is convex, \(I_{\Kscr}\) is a convex functional.
  
  For a given proper and convex functional \(\Fcal\) on \(\Lp{2}(\Omega)\), its \textit{subdifferential} at \(X \in \Lp{2}(\Omega)\) is defined as the set
  \begin{equation*}
    \del \Fcal(X) = \left\{ Y \in \Lp{2}(\Omega) \colon \Fcal(Z) - \Fcal(X) \ge \ip{Y}{Z-X} \ \forall \, Z \in \Lp{2}(\Omega) \right\}
  \end{equation*}
  For \(X \in \Kscr\), the subdifferential \(\del I_{\Kscr}(X)\) coincides with the normal cone \(N_X \Kscr\) of \(\Kscr\) at \(X\), which in analogy with the particle case \eqref{eq:NC:part} can be defined as
  \begin{equation}\label{eq:NX}
    N_X \Kscr \coloneqq \left\{ W \in \Lp{2}(\Omega) \colon \ip{W}{Y-X} \le 0 \ \forall \, Y \in \Kscr \right\}.
  \end{equation}
  Comparing \eqref{eq:proj:C} and \eqref{eq:NX} we find the following equivalence.
  For any \(X, Y \in \Lp{2}(\Omega)\) we have
  \begin{equation}\label{eq:proj:K}
    Y = \Proj{\Kscr} X \iff X-Y \in \del I_{\Kscr}(Y).
  \end{equation}
  The tangent cone \(T_X\Kscr\) of \(\Kscr\) at \(X \in \Kscr\) can then be defined similarly to the particle case \eqref{eq:TC:part} as
  \begin{equation}\label{eq:TK0}
    T_X\Kscr \coloneqq \mathrm{cl}\left\{\vartheta (Y-X) \colon Y \in \Kscr, \vartheta \ge 0 \right\}.
  \end{equation}
  Alternatively, it can be characterized as 
  the as the polar cone of the normal cone, that is,
  \begin{equation}\label{eq:TK1}
    T_X\Kscr = \left\{ U \in \Lp{2}(\Omega) \colon \ip{U}{W} \le 0 \ \forall \, W \in N_X\Kscr \right\}.
  \end{equation}
  It is convenient to introduce the following set, generalizing \(\Omega_{\bs{x}}\) in \eqref{eq:OX:part},
  \begin{equation}\label{eq:OX}
    \Omega_X \coloneqq \{ m \in \Omega \colon X \text{ is constant in a neighborhood of } m \},
  \end{equation}
  which can be thought of as the set of ``concentrated mass'' for \(\rho \in \Pscr(\R)\) corresponding to \(X_\rho = X\).
  Note that this set can be written as a countable union of disjoint intervals, \(\Omega_X = \sqcup_i (\alpha_i,\beta_i)\).
  Then we may equivalently characterize the tangent cone as follows, cf.\ \cite[Lemma 2.4]{brenier2013sticky},
  \begin{equation}\label{eq:TK2}
    T_X\Kscr = \left\{ U \in \Lp{2}(\Omega) \colon U \text{ is nondecreasing on each maximal interval } (\alpha, \beta) \subset \Omega_X \right\}.
  \end{equation}
  
  The following characterization of \(N_{X}\Kscr\) can be found in \cite[Lemma 2.3]{brenier2013sticky} and builds upon \cite[Theorem 3.9]{natile2009wasserstein}.
  \begin{lem}[Characterization of the normal cone \(N_X\Kscr\)]\label{lem:NX}
    Let \(X \in \Kscr\) and \(W \in \Lp{2}(\Omega)\) be given, and write
    \begin{equation}\label{eq:Xi}
      \Xi_W(m) \coloneqq \int_0^m W(\omega)\dee\omega \quad \text{for all } m \in [0,1]. 
    \end{equation}
    Then \(W \in N_X\Kscr\) if and only if \(\Xi_W \in \Nscr_X\), where \(\Nscr_X\) is the convex cone defined as
    \begin{equation*}
      \Nscr_X \coloneqq \left\{  \Xi \in C([0,1]) \colon \Xi \ge 0 \enspace\mathrm{in}\enspace [0,1] \enspace\mathrm{and}\enspace \Xi = 0 \enspace \mathrm{in}\enspace [0,1]\setminus\Omega_X \right\}.
    \end{equation*}
    In particular, for every \(X_1, X_2 \in \Kscr\) we have
    \begin{equation}\label{eq:OX&NX}
      \Omega_{X_1} \subset \Omega_{X_2} \implies N_{X_1}\Kscr \subset N_{X_2}\Kscr.
    \end{equation}
  \end{lem}
  An example of a function belonging to the convex cone \(\Nscr_X\) is illustrated in Figure \ref{fig:NXillust}.
  \begin{figure}
    \includegraphics[width=0.8\linewidth]{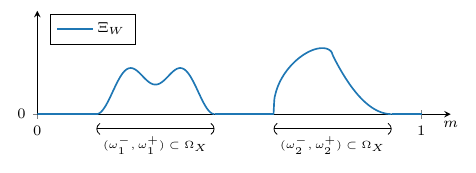}
    \caption{An \(\Xi_{W} \in \Nscr_X\) where \(\Omega_{X} = (\omega_1^-,\omega_1^+) \sqcup (\omega_2^-,\omega_2^+) \).}
    \label{fig:NXillust}
  \end{figure}
  Based on \eqref{eq:OX} we introduce the closed subspace \(\Hscr_X \subset \Lp{2}(\Omega)\) given by
  \begin{equation}\label{eq:HX}
    \Hscr_X \coloneqq \left\{ U \in \Lp{2}(\Omega) \colon U \text{ is constant on each interval } (m_l,m_r) \subset \Omega_X \right\},
  \end{equation}
  for which we have the implications
  \begin{equation}\label{eq:OX&HX}
    \Omega_{X_1} \subset \Omega_{X_2} \implies \Hscr_{X_2} \subset \Hscr_{X_1}, \qquad X_1, X_2 \in \Kscr
  \end{equation}
  and
  \begin{equation*}
    U \in \Hscr_{X} \iff \pm U \in T_{X}\Kscr.
  \end{equation*}
  We can define the projection of \(f \in \Lp{2}(\Omega)\) onto the subspace \eqref{eq:HX} as
  \begin{equation}\label{eq:proj:HX}
    \Proj{\Hscr_X}f = \begin{cases}
      \dint_{\alpha}^{\beta}f(\omega)\dee\omega \coloneqq \frac{1}{\beta-\alpha}\int_{\alpha}^{\beta}f(\omega)\dee\omega, & m \in (\alpha,\beta) \text{ a maximal interval of } \Omega_{X}, \\
      f, & \text{for a.e. } m \in \Omega\setminus\Omega_{X}.
    \end{cases}
  \end{equation}
  From \cite[Lemma 2.6]{brenier2013sticky} we have
  \begin{equation}\label{eq:HX&TX&NX}
    X \in \Kscr, \quad U \in T_{X}\Kscr \implies \Proj{\Hscr_{X}}U-U \in N_{X}\Kscr
  \end{equation}
  
  The characterization \eqref{eq:TK2} of the tangent cone \(T_{X}\Kscr\) leads us to consider the closed, convex cone defined by
  \begin{equation}\label{eq:coneab}
    \Kscr_{(\alpha,\beta)} \coloneqq \left\{ X \in \Lp{2}((\alpha,\beta)) \colon X \text{ is nondecreasing and right-continuous}\right\}
  \end{equation}
  for \((\alpha,\beta)\subset(0,1)\), such that \(\Kscr_{(0,1)} = \Kscr\).
  In \cite[Theorem 3.1]{natile2009wasserstein}, a characterization of the projection of \(f \in \Lp{2}(\Omega)\) onto \(\Kscr\) is given in terms of the right-derivative of its primitive's lower convex envelope on \(\Omega = (0,1)\).
  We recall that the lower convex envelope \(F^{**}_{(\alpha,\beta)}\) of a function \(F \in C([\alpha,\beta],\R)\) is defined as
  \begin{equation}\label{eq:lce}
    F^{**}_{(\alpha,\beta)}(m) = \sup\left\{a + b m \colon a, b \in \R, \ a + b \omega \le F(\omega) \ \forall \: \omega \in [\alpha,\beta]\right\}, \quad m \in [\alpha,\beta].
  \end{equation}
  This is the greatest bounded, (lower semi-)continuous and convex function \(G\) satisfying \(G \le F\) in \([\alpha,\beta]\); hence it is left- and right-differentiable for every \(m \in (\alpha,\beta)\), and its right-derivative \(\diff{^+}{m}F^{**}\) is nondecreasing and right-continuous.
  We note that the notation involving \(**\) alludes to the fact that the lower convex envelope in this case coincides with the biconjugate, or twice the Legendre--Fenchel transform, of \(F\).
  However, there is nothing in the proof of \cite[Theorem 3.1]{natile2009wasserstein} which relies on the domain being \((0,1)\); hence, with the appropriate changes we have the following result.
  \begin{prp}[Projection on \(\Kscr_{(\alpha,\beta)}\)]\label{prp:PK}
    Let \(f \in \Lp{2}((\alpha,\beta))\) for \((\alpha,\beta)\subset(0,1)\) and let \(F(m) = \int_{\alpha}^{m} f(\omega)\dee\omega\).
    Then
    \begin{equation}\label{eq:PK3}
      \Proj{\Kscr_{(\alpha,\beta)}}f = \diff{^+}{m}F^{**}_{(\alpha,\beta)},
    \end{equation}
    where \(F^{**}_{(\alpha,\beta)}\) is the lower convex envelope in \eqref{eq:lce}.
    Moreover, for any convex, lower semi-continuous function \(\phiv \colon \R \to (-\infty,+\infty]\) and \(f, g \in \Lp{2}((\alpha,\beta))\) we have
    \begin{equation*}
      \int_{\alpha}^{\beta} \phiv\left(\Proj{\Kscr_{(\alpha,\beta)}}f - \Proj{\Kscr_{(\alpha,\beta)}}g\right)\dee m \le \int_{\alpha}^{\beta} \phiv(f - g)\dee m.
    \end{equation*}
    In particular, \(\Proj{\Kscr_{(\alpha,\beta)}}\) is a contraction in every space \(\Lp{p}((\alpha,\beta))\), \(p \in [1,\infty]\).
  \end{prp}
  A useful auxiliary result for proving the above result is \cite[Lemma 3.2]{natile2009wasserstein}, which below is appropriately modified to match our setting.
  \begin{lem}\label{lem:PK:props}
    For any \(f \in \Lp{2}((\alpha,\beta))\) and \(F \in C([\alpha,\beta],\R)\) as defined in Proposition \ref{prp:PK}, \(F^{**}_{(\alpha,\beta)}\) is continuous on \([\alpha,\beta]\), locally Lipschitz on \((\alpha,\beta)\), and coincides with \(F\) for \(m = \alpha, \beta\). If moreover \(f \in \Lp{\infty}((\alpha,\beta))\), then \(F\) and \(F^{**}_{(\alpha,\beta)}\) are Lipschitz-continuous on \([\alpha,\beta]\).
  \end{lem}
  
  As a consequence of the above results, for any \(f \in \Lp{2}(\Omega)\) we can consider its restriction to  \((\alpha,\beta) \subset \Omega\), namely \(f\vert_{(\alpha,\beta)} \in \Lp{2}((\alpha,\beta))\), and slightly abusing notation we will still denote this by \(f\).
  Hence, when we write \(\Proj{\Kscr_{(\alpha,\beta)}}f \in \Lp{2}((\alpha,\beta))\), we mean the projection of \(f\vert_{(\alpha,\beta)}\) on \(\Kscr_{(\alpha,\beta)}\) from \eqref{eq:coneab}.
  Then, if we identify \(\Kscr = \Kscr_{(0,1)}\) and \(F^{**} = F^{**}_{(0,1)}\) in the above, we recover the original results.
  
  Now, combining the characterization \eqref{eq:TK2}, \eqref{eq:coneab} and Proposition \ref{prp:PK}, we see that for \(X \in \Kscr\) and \(f \in \Lp{2}(\Omega)\), the projection of \(f\) on \(T_X\Kscr\) can be written as
  \begin{equation}\label{eq:proj:TX}
    \Proj{T_X\Kscr}f = \begin{cases}
      \Proj{\Kscr_{(\alpha,\beta)}}f, & m \in (\alpha,\beta) \text{ a maximal interval of } \Omega_{X}, \\
      f, & \text{for a.e. } m \in \Omega\setminus\Omega_{X}.
    \end{cases}
  \end{equation}
  Let us write \(F(m) = \int_{0}^{m}f(\omega)\dee\omega\), \(F_{(\alpha,\beta)} = F(\alpha) + \int_{\alpha}^{m}f(\omega)\dee\omega\) and let \(F_{(\alpha,\beta)}^{**}\) be the lower convex envelope of \(F_{(\alpha,\beta)}\), as in \eqref{eq:lce}, such that \(F_{(\alpha,\beta)}^{**} \le F_{(\alpha,\beta)}\).
  Then it follows from \eqref{eq:PK3} and Lemma \ref{lem:PK:props} that
  \begin{equation*}
    F_{(\alpha,\beta)}(\beta) - F_{(\alpha,\beta)}(\alpha) = F_{(\alpha,\beta)}^{**}(\beta)-F_{(\alpha,\beta)}^{**}(\alpha) = \int_{\alpha}^{\beta}\Proj{\Kscr_{(\alpha,\beta)}}f \dee m.
  \end{equation*}
  Consequently, since \(\Omega_X = \sqcup_i(\alpha_i,\beta_i)\), it follows from \eqref{eq:proj:TX} that
  \begin{equation*}
    F(m) \ge \int_{0}^{m} \Proj{T_X\Kscr}f\dee \omega = \begin{cases}
      F_{(\alpha,\beta)}^{**}(m), & m \in (\alpha,\beta) \text{ a maximal interval of } \Omega_{X}, \\
      F(m), & m \in [0,1]\setminus\Omega_{X}.
    \end{cases}
  \end{equation*}
  In particular, recalling \eqref{eq:proj:HX}, we find
  \begin{equation*}
    \int_{\Omega} \Proj{\Kscr}f\dee m = \int_{\Omega} \Proj{T_X\Kscr}f\dee m = \int_{\Omega} \Proj{\Hscr_{X}}f\dee m = \int_{\Omega}f\dee m.
  \end{equation*}
  That is, although the projections onto the convex cones \(\Kscr\), \(T_X\Kscr\) and \(\Hscr_X\) are contractions in the \(\Lp{2}(\Omega)\)-norm, they do not change the average of the function being projected.
  Such considerations will turn out useful when studying clustering properties in Section \ref{s:cluster}.
  
  Figure \ref{fig:proj} provides an illustration of these projections and their primitives for a function \(f\) and a set \(\Omega_X\) consisting of three intervals on which \(f\) is respectively increasing, decreasing, and neither of the two.
  Here we can also observe the ordering of the primitives of the projections.
  \begin{figure}
    \includegraphics[width=0.8\linewidth]{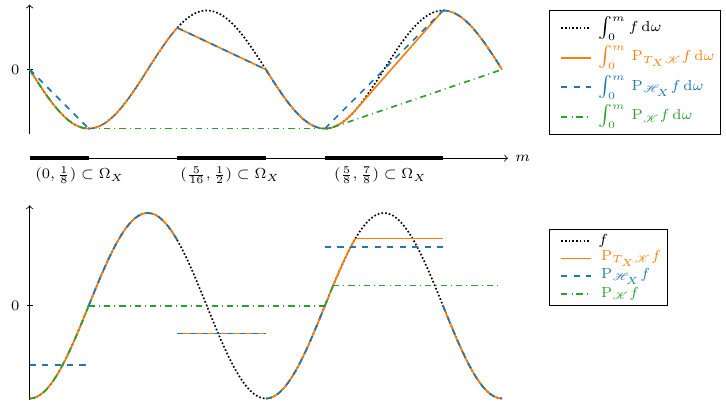}
    \caption{The various projections of \(f(m) = -4\pi\cos\left(4\pi m\right)\) and their primitives for \(\Omega_X = (0,\frac18) \sqcup (\frac{5}{16},\frac12) \sqcup (\frac58,\frac78)\). We observe that the primitive of \(\Proj{\Kscr}f\) coincides with the lower convex envelope of \(\int_{0}^{m}f(\omega)\dee\omega = -\sin(4\pi m)\) on \([0,1]\).}
    \label{fig:proj}
  \end{figure}
  
  \begin{rem}\label{rem:intCone}
    Let \(\Cscr\) be a closed, convex subset of \(\Lp{2}(\Omega)\) and \(X \in \Lp{2}((0,T);\Lp{2}(\Omega))\) such that \(X(t) \in \Cscr\) for a.e.\ \(t \in (0,T)\).
    Then Jensen's inequality yields
    \begin{equation}
      \dint_{0}^{T}X(t)\dee t \in \Cscr, \quad \text{and} \quad \int_{0}^{T}X(t)\dee t \in \Cscr \quad \text{if \(\Cscr\) is a cone.}
    \end{equation}
  \end{rem}
  
  \section{Lagrangian solutions of the Euler-alignment system}\label{s:Lag}
  We recall the assumption \(\bar{\rho} \in \Pscr_2(\R)\), for which we use equations \eqref{eq:M} and \eqref{eq:X} to define the following monotone rearrangement \(\bar{X} \in \Lp{2}(\Omega)\) satisfying \(\bar{X}_\#\mathfrak{m} = \bar{\rho}\).
  Further recalling \(\bar{v} \in \Lp{2}(\R,\bar{\rho})\) as well as \(\psi_t\) defined in \eqref{eq:psi}, we introduce the quantities \(\bar{V}\in \Lp{2}(\Omega)\) and \(\bar{\Psi} \coloneqq \Psi[\bar{X},\bar{V}]\) defined through
  \begin{equation}\label{eq:PsiXV}
    \Psi[X,V] \coloneqq V + \int_{\Omega} \Phi(X-X(\omega))\dee\omega = V + \Phi\ast\rho(X), \ \text{where} \ X_{\#}\mathfrak{m} = \rho.
  \end{equation}
  Then it also follows, recall \eqref{eq:Phi:linbnd}, that \(\bar{\Psi} \in \Lp{2}(\Omega)\).
  Here \(\bar{X}\) and \(\bar{V}\) respectively play the roles of initial position and velocity, in analogy with  \(\bar{\bs{x}}\) and \(\bar{\bs{v}}\) from Section \ref{s:particles}.
  Keeping in mind our original problem \eqref{eq:EA} and its initial data \(\bar{\rho}\) and \(\bar{v}\), a natural choice for \(\bar{V}\) is \(\bar{v}\circ\bar{X}\); we return to this point in Remark \ref{rem:initvel}.
  Likewise, the quantity \(\bar{\Psi}\), analogous to \(\bar{\bs{\psi}}\) from Section \ref{s:particles}, represents the natural velocity.
  Mimicking the arguments in the particle case, we arrive at a first-order differential inclusion corresponding to \eqref{eq:EA:DI:part},
  \begin{equation}\label{eq:EA:DI}
    \dot{X}_t + \del I_{\Kscr}(X_t) \ni \bar{\Psi} - \Phi\ast\rho_t(X_t).
  \end{equation}
  We will always use the isometry between \(\Kscr \subset \Lp{2}(\Omega)\) and \(\Pscr_2(\R)\) to have \(X_{t \#}\mathfrak{m} = \rho_t\).
  The relation \eqref{eq:EA:DI} provides us with an associated ``prescribed'' velocity
  \begin{equation}\label{eq:U}
    U[X_t] \coloneqq \bar{\Psi} - \int_\Omega \Phi(X_t-X_t(\omega))\dee\omega \in \Lp{2}(\Omega),
  \end{equation}
  such that \(U[\bar{X}] = \bar{V}\).
  Observe that \(U[X_t]\) is the velocity \(X_t \in \Kscr\) would have if \(\del I_\Kscr(X_t) = \{0\}\), i.e., \(T_{X_t}\Kscr = \Lp{2}(\Omega)\); that is, in absence of mass concentration.
  
  The next result ensures that strong convergence of \(X\) in \(\Lp{2}(\Omega)\) yields strong convergence of \(\Phi\ast\rho(X)\).
  
  \begin{lem}[Uniform continuity of \(\Phi\ast\rho\)]\label{lem:PhiConv:unicont}
    Assume \(X_i \in \Kscr\) and \(\rho_i \in \Pscr_2(\R)\) with \(\rho_i = X_{i\#}\mathfrak{m}\) for \(i\in \{1,2\}\). Then there exists a modulus of continuity \(\omega_{\Phi}\), depending only on \(\Phi\), such that
    \begin{align*}
      \norm{(\Phi\ast\rho_1)\circ X_1 - (\Phi\ast\rho_2)\circ X_2}_{\Lp{2}(\Omega)} &\le \omega_{\Phi}(\norm{X_1-X_2}_{\Lp{2}(\Omega)}),
    \end{align*}
    or equivalently, using the semidistance \(U_2\) \eqref{eq:Up} and the Kantorovich--Rubinstein metric \(W_2\) \eqref{eq:Wp},
    \begin{equation*}
      U_2((\rho_1,\Phi\ast\rho_1),(\rho_2,\Phi\ast\rho_2)) \le \omega_{\Phi}(W_2(\rho_1,\rho_2)).
    \end{equation*}
  \end{lem}
  \begin{proof}
    We will make use of the following estimate; let \(x_1, x_2 \in \R\), then for \(\rho \in \Pscr(\R)\) we find
    \begin{equation}\label{eq:convbnd}
      \begin{aligned}
        \abs{(\Phi\ast\rho)(x_2) - (\Phi\ast\rho)(x_1)} &\le \int_{\R}\abs*{\int^{x_2}_{x_1}\phi(x-y)\dee x}\dee\rho(y) \\
        &\le \int_{\R}\abs*{\int^{x_2}_{x_1}\phi\left(x-\frac{x_1+x_2}{2}\right)\dee x}\dee\rho(y) = 2 \Phi\left(\frac{\abs{x_2-x_1}}{2}\right).
      \end{aligned}
    \end{equation}
    Here the second inequality follows from \(\phi\) being symmetric and radially decreasing; indeed, the quantity \(\abs{\int^{x_2}_{x_1}\phi(x-y)\dee x}\) attains its maximal value at the midpoint \(y = \frac12(x_1+x_2)\).
    Now we estimate
    \begin{equation*}
      \norm{(\Phi\ast\rho_1)\circ X_1 - (\Phi\ast\rho_2)\circ X_2}_{\Lp{2}(\Omega)}^2 \le 2(I_1 + I_2) 
    \end{equation*}
    where
    \begin{equation*}
      I_1 \coloneqq \norm{(\Phi\ast\rho_1)\circ X_2 - (\Phi\ast\rho_2)\circ X_2}_{\Lp{2}(\Omega)}^2, \qquad I_2 \coloneqq \norm{(\Phi\ast\rho_1)\circ X_1 - (\Phi\ast\rho_1)\circ X_2}_{\Lp{2}(\Omega)}^2, \qquad .
    \end{equation*}
    Then, from \eqref{eq:convbnd}, the concavity of \(\Phi\) and Jensen's inequality we obtain
    \begin{align*}
      I_1 &\le \int_{\Omega}\left(\int_{\Omega} \abs*{\Phi(X_2(m)-X_1(\omega))-\Phi(X_2(m)-X_2(\omega))}\dee\omega\right)^2\dee m \\
      &\le \int_{\Omega}\left(2\int_{\Omega}\Phi\left(\frac{\abs{X_1(\omega)-X_2(\omega)}}{2}\right)\dee\omega\right)^2\dee m \le 4 \Phi\left(\frac12\norm{X_1-X_2}_{\Lp{1}(\Omega)}\right)^2.
    \end{align*}
    On the other hand, defining \(\Omega_- \coloneqq \{m \in \Omega \colon \abs{X_1-X_2} \le 2\}\) and  \(\Omega_+ \coloneqq \Omega\setminus\Omega_-\), the same inequalities together with \eqref{eq:Phi:linbnd} yield
    \begin{align*}
      I_2 &\le 4 \int_{\Omega} \Phi\left(\frac12\abs{X_1(m)-X_2(m)}\right)^2\dee m \\
      &\le 4\Phi(1) \int_{\Omega_-}\Phi\left(\frac12\abs{X_1(m)-X_2(m)}\right)\dee m + 4\int_{\Omega_+}\left(\Phi(1)\frac12\abs{X_1(m)-X_2(m)}\right)^2\dee m \\
      &\le 4\Phi(1) \Phi\left(\frac12\norm{X_1-X_2}_{\Lp{1}(\Omega)}\right) + \left(\Phi(1)\norm{X_1-X_2}_{\Lp{2}(\Omega)}\right)^2.
    \end{align*}
    Since \(\mathfrak{m}(\Omega) = 1\), we have \(\norm{X}_{\Lp{1}(\Omega)}\le\norm{X}_{\Lp{2}(\Omega)}\), and the result follows from collecting all the bounds,
    \begin{align*}
      \norm{(\Phi\ast\rho_1)\circ X_1 - (\Phi\ast\rho_2)\circ X_2}_{\Lp{2}(\Omega)}^2 &\le 8\Phi\left(\max\left\{1,\frac12\norm{X_1-X_2}_{\Lp{2}(\Omega)}\right\}\right)\Phi\left(\frac12\norm{X_1-X_2}_{\Lp{2}(\Omega)}\right) \\
      &\quad+ \left(\Phi(1)\norm{X_1-X_2}_{\Lp{2}(\Omega)}\right)^2.
    \end{align*}
    That we in the end obtain a modulus of continuity is then a consequence of \(\Phi\) being such a function, cf.\ Lemma \ref{lem:Phi:props} (b).
  \end{proof}
  
  \subsection{The associated functional and gradient flow structure}
  Next we will see how \eqref{eq:EA:DI} can be regarded as a gradient flow for a certain convex and lower semi-continuous functional.
  Based on the gradient flow structure in the particle case, cf.\ \eqref{eq:pot:part}, we are led to consider the following functional
  \begin{equation}\label{eq:pot}
    \Vcal(X) = \Wcal_{\Phi}(X) - \ip{\bar{\Psi}}{X} \coloneqq \frac12 \iint_{\Omega\times\Omega} W_{\Phi}(X(m)-X(\omega)) \dee\omega \dee m -\int_\Omega \bar{\Psi}(m)X(m)\dee m,
  \end{equation}
  where \(W_{\Phi}\) is the primitive of the odd function \(\Phi\) from \eqref{eq:Phi} satisfying \(W_{\Phi}(0) = 0\), i.e., the even function
  \begin{equation}\label{eq:PhiPrim}
    W_{\Phi}(x) \coloneqq \int_0^x \Phi(y)\dee y
  \end{equation}
  illustrated in Figure \ref{fig:kernels}.
  This function is also convex on the whole of \(\R\) due to \(\Phi(x)\) being nondecreasing, and its second derivative is the nonnegative communication protocol \(\phi\).
  We observe that \(\Wcal_{\Phi}\) is a proper functional for \(X \in \Lp{2}(\Omega)\), as it can be bounded with the pointwise linear bound \eqref{eq:Phi:linbnd}.
  The second term of \(\Vcal\) is linear in \(X\), and hence bounded by Cauchy--Schwarz and \(\bar{\Psi} \in \Lp{2}(\Omega)\).
  We conclude that \(\Vcal\) is a proper functional with \(\dom(\Vcal) = \Lp{2}(\Omega)\).
  One can check that the G\^{a}teaux derivative \(\nabla_X \Vcal(X)\) of \(\Vcal\) is exactly the negative of the right-hand side in \eqref{eq:EA:DI}, that is, \(-U[X]\) from \eqref{eq:U}, owing to the oddness of \(\Phi\).
  That is,
  \begin{equation*}
    \nabla_X \Wcal_{\Phi}(X) = \Phi\ast\rho(X), \qquad \nabla_{X}\Vcal(X) = \Phi\ast\rho(X) - \bar{\Psi} = -U[X].
  \end{equation*}
  We would like to show that \(\Vcal(X)\) is convex in \(X \in \Kscr\).
  To this end, for \(Y_0, Y_1 \in \Kscr\) and \(\vartheta \in (0,1)\) we introduce the convex combination \(Y_{\vartheta} \coloneqq (1-\vartheta)Y_0 + \vartheta Y_1\).
  Our goal is then to show
  \begin{equation*}
    \Vcal(Y_\vartheta) \le (1-\vartheta)\Vcal(Y_0) + \vartheta \Vcal(Y_1).
  \end{equation*}
  The second term of \eqref{eq:pot} is already linear in \(X\), and so we turn to the first term.
  Observe next that we can write the argument of the inner integrand as
  \begin{equation*}
    Y_\vartheta(m) - Y_\vartheta(\omega) = (1-\vartheta)[Y_0(m)-Y_0(\omega)] + \vartheta [Y_1(m)-Y_1(\omega)].
  \end{equation*}
  Then it follows directly from the convexity of \(W_{\Phi}\) in \eqref{eq:PhiPrim} that also the first term of \eqref{eq:pot} is convex in \(X \in \Kscr\), and we conclude that the entire functional \(\Vcal(X)\) is convex in \(X \in \Kscr\).
  
  Now, since \(\Vcal(X)\) is convex for \(X \in \Kscr\) and has G\^{a}teaux derivative \(\nabla_X\Vcal(X) = -U[X]\), it follows from \cite[Proposition 3.20]{peypouquet2015convex} that its subdifferential then reduces to a single element.
  \begin{prp}
    For \(X \in \Kscr\), the subdifferential \(\del \Vcal(X)\) of \(\Vcal\) defined in \eqref{eq:pot} consists of a single element.
    In particular, we have \(\del \Vcal(X) = \{-U[X]\} \) for \(U[X]\) defined in \eqref{eq:U}.
  \end{prp}
  
  In order to ensure that our flow remains in the cone \(\Kscr\), we have to consider a modified version of the functional \(\Vcal\), namely
  \begin{equation}\label{eq:potK}
    \bar{\Vcal}(X) \coloneqq \Vcal(X) + I_{\Kscr}(X),
  \end{equation}
  where we recall the indicator function \(I_{\Kscr}\) from \eqref{eq:indicator}.
  We would like to show that this is convex for \(X \in \Lp{2}(\Omega)\); that is, considering again the convex combination \(Y_\vartheta\) from before, now with \(Y_0, Y_1 \in \Lp{2}(\Omega)\), we want to show
  \begin{equation}\label{eq:convexity}
    \bar{\Vcal}(Y_\vartheta) \le (1-\vartheta)\bar{\Vcal}(Y_0) + \vartheta \bar{\Vcal}(Y_1).
  \end{equation}
  Suppose either of \(Y_0\), \(Y_1\) does not lie in the cone \(\Kscr\); then according to \eqref{eq:indicator} and \eqref{eq:potK}, the right-hand side of \eqref{eq:convexity} is infinite, and the inequality is trivially satisfied.
  It remains to consider the case when both \(Y_0\) and \(Y_1\) belong to \(\Kscr\).
  In this case the indicator function vanishes due to the convexity of \(\Kscr\), such that \(\bar{\Vcal}\) reduces to \(\Vcal\), which is the case we treated before.
  Moreover, since \(\Vcal\) is continuous and \(\Kscr\) is closed, the functional \(\bar{\Vcal}\) is lower semi-continuous.
  We have therefore proved the following result, which will be useful for establishing existence and uniqueness of solutions to \eqref{eq:EA:DI}.
  \begin{prp}
    The functional \(\bar{\Vcal}\) in \eqref{eq:potK} is convex and lower semi-continuous in \(\Lp{2}(\Omega)\).
  \end{prp}
  Like in the proof of \cite[Theorem 2.13]{bonaschi2015equivalence}, we will make use of the fact that the subdifferential \(\del \bar{\Vcal} = \del(\Vcal + I_{\Kscr})\) is additive in this case.
  Indeed, since \(\Vcal\) is proper, lower semi-continuous and convex, as well as bounded on the whole domain of \(\dom(I_{\Kscr}) = \Kscr\), this follows from the Moreau--Rockafellar theorem \cite[Theorem 3.30]{peypouquet2015convex}, see also the the proof of \cite[Proposition 3.61]{peypouquet2015convex}.
  \begin{prp}
    For \(X \in \Kscr\), the subdifferential \(\del \bar{\Vcal}(X)\) of \(\bar{\Vcal}\) in \eqref{eq:potK} can be written as
    \begin{equation*}
      \del \bar{\Vcal}(X) = \del I_{\Kscr}(X) - U[X] = \del I_{\Kscr}(X) - \bar{\Psi} + \int_{\Omega} \Phi(X-X(\omega))\dee\omega.
    \end{equation*}
  \end{prp}
  Now we will provide the definition of an \(\Lp{2}\)-gradient flow solution associated with the functional \eqref{eq:potK}, cf. \cite[Definition 2.11]{bonaschi2015equivalence}.
  This will justify the differential inclusion \eqref{eq:EA:DI} for the Euler-alignment system.
  
  \begin{dfn}[\(\Lp{2}\)-gradient flow]\label{dfn:gradflow}
    An absolutely continuous curve \(X_t \colon [0,\infty) \to \Lp{2}(\Omega) \) is an \(\Lp{2}\)-gradient flow for the functional \(\bar{\Vcal}\) in \eqref{eq:potK} if it satisfies the differential inclusion
    \begin{equation}\label{eq:gradflow}
      -\dot{X}_t \in \del \bar{\Vcal}(X_t) = -U[X_t] + \del I_{\Kscr}(X_t) \quad \text{for a.e.} \ t > 0.
    \end{equation}
  \end{dfn}
  Observe that \eqref{eq:gradflow} is exactly \eqref{eq:EA:DI}, which was derived heuristically from the particle case in Section \ref{s:particles}.
  
  \subsection{Lagrangian solutions}
  Since \(\Lp{2}(\Omega)\) is a Hilbert space, we will, like the works \cite{natile2009wasserstein,brenier2013sticky,bonaschi2015equivalence}, rely on the theory of \cite{brezis1973operateurs} to establish existence and uniqueness of a solution \(X_t\) evolving according to the minimal element of the subdifferential \(\del\bar{\Vcal}(X_t)\).
  This element is sometimes called the principal section, and we denote it by \(\del^\circ \bar{\Vcal}(X_t)\).
  The subdifferential \(\del\bar{\Vcal}\) is a maximally monotone operator, cf.\ \cite[Exemple 2.3.4]{brezis1973operateurs}, and we can apply \cite[Theoreme 3.1]{brezis1973operateurs} to deduce existence and uniqueness of such solutions to \eqref{eq:gradflow}.
  In fact, since \(\bar{\Vcal}\) is proper, lower semi-continuous and convex we can directly apply \cite[Theoreme 3.2]{brezis1973operateurs} to obtain the following result.
  \begin{thm}[Lagrangian solution]\label{thm:LagSol}
    Let \(\bar{X} \in \Kscr\) and \(\bar{V} \in \Lp{2}(\Omega)\), such that \(\bar{\Psi} = \Psi[\bar{X},\bar{V}]\) given by \eqref{eq:PsiXV}. Then there exists a unique gradient flow solution \(X_t \colon [0,\infty) \to \Kscr\) in the sense of Definition \ref{dfn:gradflow} with initial data \(\bar{X}\) which evolves according to the minimal selection of \(\del \bar{\Vcal}\), that is,
    \begin{equation*}
      -\dot{X}_t \in \del^\circ \bar{\Vcal}(X_t) \quad \text{for a.e.} \ t > 0.
    \end{equation*}
    We will call this solution the Lagrangian solution associated with the Euler-alignment system.
  \end{thm}
  
  From the same theory, the Lagrangian solution in Theorem \ref{thm:LagSol} enjoys many useful properties, and following \cite[Theorem 3.5]{brenier2013sticky}, we list some of them here.
  \begin{cor}[Properties of the Lagrangian solution]\label{cor:LagSol:props}
    Let \(X_t\) be the Lagrangian solution from Theorem \ref{thm:LagSol}, with corresponding prescribed velocity \(U[X_t]\) as defined in \eqref{eq:U}. Then it follows:
    \begin{enumerate}[label={\upshape(\alph*)}]
      \item The right-derivative \(\diff{^+}{t}X_t \eqqcolon V_t\) exists for all \(t \ge 0\).
      \item The velocity \(V_t\) is the minimal element of the subdifferential, that is,
      \begin{equation}\label{eq:V:minimal}
        V_t =  \left( U[X_t] -\del I_{\Kscr}(X_t) \right)^\circ.
      \end{equation}
      In particular, \eqref{eq:EA:DI} holds true for all \(t \ge 0\) if we replace \(\dot{X}_t\) with \(V_t\).
      \item The velocity \(V_t\) is the projection of \(U[X_t]\) on the tangent cone \(T_{X_t}\Kscr\):
      \begin{equation}\label{eq:V:PTXU}
        V_t = \Proj{T_{X_t}\Kscr} U[X_t], \qquad \Psi_t \coloneqq V_t + \Phi\ast\rho_t(X_t) = \Proj{T_{X_t}\Kscr} \bar{\Psi} \quad \text{for all } t \ge 0.
      \end{equation}
      \item Continuity of the velocity: \(t \mapsto V_t\) is right-continuous for all \(t \ge 0\), in particular
      \begin{equation}\label{eq:V:rightcont}
        \lim\limits_{t\to0+}V_t = \bar{V} \quad \text{if and only if} \quad \bar{V} \in T_{\bar{X}}\Kscr.
      \end{equation}
      Moreover, the map \(t \mapsto \norm{V_t}_{\Lp{2}(\Omega)}\) is nonincreasing, and so there is an at most countable set of times \(\Tcal \subset (0,\infty)\) where it is discontinuous.
      Then \(t \mapsto X_t\) is continuously differentiable in \((0,\infty)\setminus\Tcal\). Defining \(\rho_t \coloneqq X_{t \#}\mathfrak{m} \in \Pscr_2(\R)\), there exists a unique map \(v_t \in \Lp{2}(\R,\rho_t)\) such that
      \begin{equation}\label{eq:V:comp}
        \dot{X}_t = V_t = \Proj{\Hscr_{X_t}}U[X_t] = v_t \circ X_t \in \Hscr_{X_t} \ \text{for every } t \in (0,\infty)\setminus\Tcal.
      \end{equation}
      \item Dissipation identity:
      \begin{equation}\label{eq:diss}
        \diff{^+}{t} \Vcal(X_t) = -\norm{V_t}_{\Lp{2}(\Omega)}^2 \quad \text{for \(t > 0\), implying} \quad \Vcal(X_s) - \Vcal(X_t) = \int_s^t \norm{V_r}_{\Lp{2}(\Omega)}^2\dee r
      \end{equation}
      \item Stability estimates: Let \((X^i_t, V_t^i)\) be Lagrangian solutions of \eqref{eq:EA:DI} with respective initial data \((\bar{X}^{i},\bar{V}^{i})\), such that \(\bar{\Psi}^{i} = \Psi[\bar{X}^{i},\bar{V}^{i}]\). Then we have
      \begin{equation}\label{eq:X:stab}
        \norm{X_t^1-X_t^2}_{\Lp{2}(\Omega)} \le \norm{X_s^1-X_s^2}_{\Lp{2}(\Omega)} + (t-s)\norm{\bar{\Psi}^1-\bar{\Psi}^2}_{\Lp{2}(\Omega)}
      \end{equation}
      for \(s \le t\), as well as
      \begin{equation}\label{eq:V:stab}
        \int_0^T \norm{V_t^1-V_t^2}^2_{\Lp{2}(\Omega)}\dee t \le C\left(\sum_{i=1}^{2}(\norm{\bar{X}^{i}} + \norm{\bar{V}^{i}} + \sqrt{T} \norm{\bar{\Psi}^{i}})\right) \left( \norm{\bar{X}^1-\bar{X}^2}_{\Lp{2}(\Omega)} + \sqrt{T}\norm{\bar{\Psi}^1 - \bar{\Psi}^2}_{\Lp{2}(\Omega)} \right),
      \end{equation}
      for some constant \(C > 0\) independent of \(T\) and the initial data. 
    \end{enumerate}
  \end{cor}
  \begin{proof}
    Properties (a), (b) and the right-continuity of \(V_t\) are consequences of \cite[Theoreme 3.1]{brezis1973operateurs}, while (e) follows from \cite[Theoreme 3.2]{brezis1973operateurs}.
    
    By property (b), \(V_t\) is the \(\Lp{2}\)-projection of the zero function onto the closed, convex subset \(U[X_t]-\del I_{\Kscr}(X_t)\), and so by \eqref{eq:proj:C} it follows that \(\ip{V}{U-V-\xi} \ge 0\) for all \(\xi \in \del I_{\Kscr}(X_t)\).
    Furthermore, since \(U[X_t]-V_t\) belongs to the cone \(\del I_{\Kscr}(X_t)\), it follows that both \(0\) and \(2(U[X_t]-V_t)\) belong to  \(\del I_{\Kscr}(X_t)\); hence by the previous inequality \(U[X_t]-V_t \perp V_t\).
    On the other hand, by property (a) and \eqref{eq:TK0} it is clear that \(V_t \in T_{X_t}\Kscr\). Then it follows from the characterization \eqref{eq:proj:C} that \(V_t = \Proj{T_{X_t}\Kscr}U[X_t]\).
    Since \(\Phi\ast\rho_t(X_t) \in \Hscr_{X_t}\), it is clear from \eqref{eq:TK2} that \(\Proj{T_{X_t}\Kscr}(U[X_t]) = \Proj{T_{X_t}\Kscr}\bar{\Psi} - \Phi\ast\rho_t(X_t)\).
    
    The stability estimates (f) follow from instead considering \eqref{eq:gradflow} from the point of view of the proper, convex and lower semi-continuous functional \(\bar{\Wcal}_{\Phi}(X) = \Wcal_{\Phi}(X) + I_{\Kscr}(X)\), and that \eqref{eq:gradflow} is equivalent to \(\dot{X}_t + \del \bar{\Wcal}_{\Phi}(X_t) \ni \bar{\Psi}\).
    Then \eqref{eq:X:stab} follows from \cite[Lemme 3.1]{brezis1973operateurs}.
    On the other hand, \eqref{eq:V:stab} is a consequence of \cite[Theorem 2]{savare1996weak}; in their notation, our setting translates to the Hilbert triple \(V = H = V' = \Lp{2}(\Omega)\) and quantities \(w(u_0,f) = \norm{\bar{X}}_{\Lp{2}(\Omega)} + \norm{\bar{V}}_{\Lp{2}(\Omega)}\) and \(\norm{f}_{H^1(0,T;V')}^2 = T\norm{\bar{\Psi}}^2_{\Lp{2}(\Omega)} \).
  \end{proof}
  
  \begin{rem}[The initial velocity]\label{rem:initvel}
    Corollary \ref{cor:LagSol:props} allows us to have \(\bar{V} \in T_{\bar{X}}\Kscr \setminus \Hscr_{\bar{X}} \), and still have the right-derivative \(V_t \to \bar{V}\) as \(t\to0+\);
    this would then mean that the initial concentrated mass has different initial velocities.
    However, if our aim is to study \eqref{eq:EA} with initial data \(\bar{\rho}\) and \(\bar{v} \in \Lp{2}(\R,\bar{\rho})\), this does not make much sense, as the natural choice of \(\bar{V}\) would be \(\bar{v}\circ\bar{X} \in \Hscr_{\bar{X}}\).
  \end{rem}
  \begin{rem}[Energy dissipation]
    We will not directly make use of the energy dissipation relation \eqref{eq:diss} in this study, but merely point out that energy dissipation plays a role in a recent study of a kinetic Cucker--Smale model \cite[Section 3]{peszek2022measure}.
  \end{rem}
  
  We can think of picking the minimal element in the set \(U[X_t]-\del I_{\Kscr}(X_t)\), cf.\ \eqref{eq:V:minimal}, as finding the velocity \(V_t\) closest to the prescribed velocity \(U[X_t]\) which still allows \(X_t\) to remain in \(\Kscr\).
  Note that in \eqref{eq:V:PTXU} we have, in analogy with the discrete quantity \(\bs{\psi}(t)\) from Section \ref{s:particles}, introduced \(\Psi_t\), which satisfies the corresponding continuum version of the discrete relation \eqref{eq:proj:psi}.
  That is, it is the projection of the natural velocity \(\bar{\Psi}\) on the tangent cone \(T_{X_t}\Kscr\).
  
  Observe that \(\frac{1}{h}(X_{t-h}-X_t) \in T_{X_t}\Kscr \) and is uniformly bounded for any for any \(0 < h < t\), hence it has a subsequence converging weakly to \(-V' \in T_{X_t}\Kscr\) by the definition \eqref{eq:TK1} and the cone being weakly closed.
  Then, by the characterization \eqref{eq:TK2}, any such weak limit must be nonincreasing on connected components of \(\Omega_{X_t}\), which is analogous to the monotonicity relation leading to the barycentric lemma \eqref{eq:barycentric} in the particle case.
  
  If the relation \(V_t \in \Hscr_{X_t}\) in \eqref{eq:V:comp} were to hold for all times, this is one of the hallmarks of the so-called sticky Lagrangian solutions, and we recall their definition \cite[Definition 3.7]{brenier2013sticky} next.
  \begin{dfn}[Sticky Lagrangian solutions]\label{dfn:LagSol:sticky}
    A Lagrangian solution is called \textit{sticky} if
    \begin{equation}\label{eq:sticky:OX}
      \text{for any} \ s \le t \ \text{we have} \ \Omega_{X_{s}} \subset \Omega_{X_{t}}.
    \end{equation}
  \end{dfn}
  Because of the implications \eqref{eq:OX&NX} and \eqref{eq:OX&HX},
  any sticky Lagrangian solution satisfies the monotonicity condition
  \begin{equation}\label{eq:sticky:NX&HX}
    \del I_{\Kscr}(X_s) \subset \del I_{\Kscr}(X_t), \quad \Hscr_{X_t} \subset \Hscr_{X_s} \enspace \text{for any } s \le t.
  \end{equation}
  Another property of sticky solutions which will prove useful below is the following,
  \begin{equation}\label{eq:sticky:PHXs&NXt}
    \xi \in \del I_{\Kscr}(X_t) \implies \Proj{\Hscr_{X_s}}\xi \in \del I_{\Kscr}(X_t).
  \end{equation}
  To prove this we will use the characterization from Lemma \ref{lem:NX}.
  Let \(\Xi\) be the corresponding primitive \eqref{eq:Xi} of \(\xi\), then we know that \(\Xi(m) = 0\) for \(m \in \Omega\setminus\Omega_{X_t}\), and \(\Xi(m) \ge 0\) for \(m \in \Omega_{X_t}\).
  By \eqref{eq:sticky:NX&HX}, \(\Proj{\Hscr_{X_s}}\) will only modify \(\xi\) in \(\Omega_{X_s} \subset \Omega_{X_t}\), so let us consider a maximal interval \((\alpha_s,\beta_s) \subset \Omega_{X_s}\) which is necessarily contained in a maximal interval \((\alpha_t,\beta_t) \subset \Omega_{X_t}\).
  On \((\alpha_s, \beta_s)\), the projection will replace \(\xi\) with its average value over this interval, which means that the corresponding primitive of \(\Proj{\Hscr_{X_s}}\xi\) on this interval will be \(\Xi(\alpha_s) + \frac{m-\alpha_s}{\beta_s-\alpha_s}(\Xi(\beta_s)-\Xi(\alpha_s)) \ge 0\).
  An analogous expression holds on any other maximal interval of \(\Omega_{X_s}\) contained in \((\alpha_t,\beta_t)\), while on the remaining part \((\alpha_t,\beta_t)\setminus\Omega_{X_s}\) its value is \(\Xi(m) \ge 0\).
  Hence, the primitive of \(\Proj{\Hscr_{X_s}}\xi\) also belongs to \(\Nscr_{X_t}\), meaning \eqref{eq:sticky:PHXs&NXt} holds.
  
  Sticky Lagrangian solutions satisfy additional properties, and analogous to \cite[Proposition 3.8]{brenier2013sticky} we showcase some of these below.
  \begin{prp}[Projection formulas for sticky Lagrangian solutions]\label{prp:sticky:projection}
    Let \(X_t \colon [0,\infty) \to \Kscr\) be a sticky Lagrangian solution of \eqref{eq:EA:DI} in the sense of Definition \ref{dfn:LagSol:sticky}.
    Then
    \begin{equation}\label{eq:sticky:HX}
      V_t \in \Hscr_{X_t}, \ \text{and so also} \ \Psi_t \in \Hscr_{X_t}, \ \text{for all} \ t \ge 0.
    \end{equation}
    Furthermore, for \(0 \le s \le t\), \(X_t\) and \(\Psi_t\) satisfy
    \begin{equation}\label{eq:sticky:Psi}
      \Psi_s - \Psi_t \in \del I_{\Kscr}(X_t) \implies \Psi_t = \Proj{\Hscr_{X_t}}\Psi_s,
    \end{equation}
    \begin{equation}\label{eq:sticky:Xs}
      X_t = \Proj{\Kscr} \left( X_s + (t-s)\Psi_s - \int_s^t \Phi\ast\rho_r(X_r) \dee r \right).
    \end{equation}
    In particular, we have the formulas
    \begin{equation}\label{eq:sticky:V}
      V_t = \Proj{\Hscr_{X_t}}U[X_t] = \Proj{\Hscr_{X_t}}\bar{\Psi} - \Phi\ast\rho_t(X_t), \ \text{and so also} \ \Psi_t = \Proj{\Hscr_{X_t}}\bar{\Psi},
    \end{equation}
    \begin{equation}\label{eq:sticky:X}
      X_t = \Proj{\Kscr} \left( \bar{X} + \int_0^t U[X_s]\dee s \right) = \Proj{\Kscr} \left( \bar{X} + t\bar{\Psi} - \int_0^t \Phi\ast\rho_s(X_s) \dee s \right).
    \end{equation}
  \end{prp}
  \begin{proof}
    The inclusion \eqref{eq:sticky:HX} follows from the right-continuity of \(V_t\) and \eqref{eq:V:comp} in Corollary \ref{cor:LagSol:props} together with the monotonicity property \eqref{eq:sticky:NX&HX}.
    The identity \eqref{eq:sticky:V} follows from applying the projection to \eqref{eq:EA:DI} with \(\dot{X}_t\) replaced with \(V_t\).
    As a consequence of \eqref{eq:sticky:V} and \eqref{eq:sticky:NX&HX} we then have \(\Psi_t = \Proj{\Hscr_{X_t}}\Psi_s\).
    On the other hand, \(\Proj{\Hscr_{X_s}}\Psi_t = \Psi_t\) since \(\Psi_t\) is already constant where \(X_s\) is constant.
    Then, from \eqref{eq:EA:DI} it follows that \(\bar{\Psi}-\Psi_t \in \del I_{\Kscr}(X_t)\), and applying \(\Proj{\Hscr_{X_s}}\) to the left-hand side we obtain \(\Psi_s-\Psi_t\), which means that \eqref{eq:sticky:Psi} follows from \eqref{eq:sticky:PHXs&NXt};
    this also implies \(\Psi_t = \Proj{\Hscr_{X_t}}\Psi_s\), since \(\del I_{\Kscr}(X_t) \subset \Hscr_{X_t}^\perp\).
    
    We then note that for \(0 \le s \le r \le t\) we have
    \begin{equation*}
      \Psi_s - \Phi\ast\rho_r(X_r) - V_r = \Psi_s - \Psi_r \in \del I_{\Kscr}(X_r) \subset \del I_{\Kscr}(X_t),
    \end{equation*}
    and integrating this expression we obtain, cf.\ Remark \ref{rem:intCone},
    \begin{equation*}
      (t-s)\Psi_s -\int_s^t \Phi\ast\rho_r(X_r)\dee r + X_s - X_t \in \del I_{\Kscr}(X_t).
    \end{equation*}
    Then \eqref{eq:sticky:Xs} is a consequence of \eqref{eq:proj:K}, while \eqref{eq:sticky:X} follows from \(s = 0\) and \(\Psi_0 = \bar{\Psi}\). 
  \end{proof}
  \begin{rem}[Semigroup property]\label{rem:semig}
    From \eqref{eq:sticky:Psi} and \eqref{eq:sticky:Xs} we see that the sticky Lagrangian solution is a semigroup \(\hat{S}_t \colon (\bar{X},\bar{\Psi}) \to (X_t, \Psi_t)\).
    Furthermore, since \(\bar{V} = \bar{\Psi}-\Phi\ast\bar{\rho}(\bar{X})\), \(\rho_t = X_{t \#}\mathfrak{m}\), and \(V_t = \Psi_t - \Phi\ast\rho_t(X_t)\), this shows that also \(\check{S}_t \colon (\bar{X},\bar{V}) \to (X_t, V_t)\) is a semigroup.
  \end{rem}
  \begin{rem}[Reduction to pressureless Euler]\label{rem:proj:NS}
    In the case \(\phi \equiv 0\) we have \(\Phi \equiv 0\) and \(\bar{\Psi} = \bar{V}\), thereby recovering the projection formulas in \cite[Theorem 2.6]{natile2009wasserstein} for the pressureless Euler system.
  \end{rem}
  The formulas of Proposition \ref{prp:sticky:projection} could be useful for numerical implementations of the Euler-alignment system.
  For instance, in the case of pressureless Euler, i.e., \(\phi \equiv 0\), \cite[Section 4]{brenier1998sticky} presents a computational algorithm based on the convex envelope, which is connected to the projection \(\Proj{\Kscr}\) through Proposition \ref{prp:PK}, see \cite[Theorem 6.1]{natile2009wasserstein} for details.
  However, in our case such formulas are less straightforward, since \eqref{eq:sticky:X} depends on \(X_s\) for \(0 \le s \le t\);
  we will return to this point in Remark \ref{rem:proj:BGSW}, after we have established that our Lagrangian solutions are indeed sticky.
  To this end we will make use of a natural framework for a numerical method in this setting, namely particle approximations, which were the foundation for the front tracking approximations in \cite{leslie2023sticky}.
  
  \subsection{Sticky particle solutions}
  Since the dynamics derived in \cite{leslie2023sticky} is based on sticky particles, and both \cite{natile2009wasserstein} and \cite{brenier2013sticky} use limits of particle solutions to deduce sticky behavior, we will do the same here.
  Let us recall the motivation of the particle solutions from Section \ref{ss:partintro}, in particular we had \(\bs{m} \in \R^N_+\).
  For convenience, following \cite{brenier2013sticky}, we introduce the set \(\M^N\) defined as
  \begin{equation*}
    \M^N \coloneqq \left\{\bs{m} \in \R_+^N \colon \sum_{i=1}^{N}m_i = 1 \right\}.
  \end{equation*}
  For \(\bs{m} \in \M^N\) we can define a partition of \([0,1)\) through the quantities
  \begin{equation}\label{eq:thetai}
    \theta_0 \coloneqq 0, \qquad \theta_i \coloneqq \sum_{j=1}^{i}m_j, \quad i \in \{1,\dots,N\}.
  \end{equation}
  Let us denote by \(\chi_{A}\) the characteristic function of a subset \(A \subset [0,1]\).
  Then we can define the finite-dimensional Hilbert space
  \begin{equation}\label{eq:Hm}
    \Hscr_{\bs{m}} \coloneqq \left\{ X = \sum_{i=1}^{N} x_i \chi_{[\theta_{i-1}, \theta_i)} \colon (x_1,\dots,x_N) = \bs{x} \in \R^N \right\} \subset \Lp{2}(\Omega)
  \end{equation}
  and its convex cone
  \begin{equation}\label{eq:Km}
    \Kscr_{\bs{m}} \coloneqq \left\{ X = \sum_{i=1}^{N} x_i \chi_{[\theta_{i-1}, \theta_i)} \colon (x_1,\dots,x_N) = \bs{x} \in \K^N \right\} \subset \Kscr \subset \Lp{2}(\Omega).
  \end{equation}
  For \(X^N \in \Kscr_{\bs{m}}\) defined through \(\bs{x} \in \K^N\) and \(\bs{m} \in \M^N\), we can introduce the empirical measure \(\rho^N_{\bs{x}}\) and its cumulative distribution \(M^N\),
  \begin{equation*}
    \rho^N_{\bs{x}} = \sum_{i=1}^{N}m_i \delta_{x_i}, \qquad M^N(x) = \sum_{i=1}^{N}m_i H(x-x_i) = \sum_{i=1}^{N}\theta_i \chi_{[x_{i},x_{i+1})}(x),
  \end{equation*}
  where \(H(x)\) is the right-continuous Heaviside function.
  Then we know from before that \(M^N\) is the generalized inverse of \(X^N\).
  Moreover, we see that convolution of \(\Phi\) with \(\rho^N_{\bs{x}}\) yields exactly the ``discrete convolution'' \(\Phi_{\bs{m}}^\ast\) from Section \ref{s:particles},
  \begin{equation*}
    (\Phi\ast\rho^N)(x) = \sum_{j=1}^{N}m_j\Phi(x-x_j) = \sum_{j=1}^{N}\int_{\theta_{j-1}}^{\theta_j}\Phi(x-X^N(\omega))\dee\omega = \int_{\Omega} \Phi(x-X^N(\omega))\dee\omega.
  \end{equation*}
  
  \begin{rem}
    Let \(\bs{x},\bs{y} \in \K^N\) and respectively define \(X, Y \in \Kscr_{\bs{m}}\) for \(\bs{m} \in \M^N\) and the empirical measures \(\rho_{\bs{x}}, \rho_{\bs{y}} \in \Pscr(\R)\).
    Then we have the following identities
    \begin{equation*}
      W_p(\rho_{\bs{x}}, \rho_{\bs{y}}) = \norm{X-Y}_{\Lp{p}(\R)} = \norm{\bs{x}-\bs{y}}_{\bs{m},p},
    \end{equation*}
    that is, we recover the \(\bs{m}\)-weighted Euclidean \(p\)-norm from \eqref{eq:euclid}.
  \end{rem}
  
  Next we want to formulate the particle dynamics from Section \ref{s:particles} using the above framework.
  For a given \(\bs{m} \in \M^N\), suppose \(\bar{X}^N \in \Kscr_{\bs{m}}\) and \(\bar{V}^N \in \Hscr_{\bar{X}^N}\).
  We may without loss of generality assume the initial particles to be separate, i.e., \(\bar{\bs{x}} \in \mathrm{int}(\K^N)\); otherwise, we could reduce and relabel the particles.
  Then we may take \(\bar{\bs{v}} \in \R^N\), since \(T_{\bar{\bs{x}}}\K^N = \R^N\), and we define
  \begin{equation}\label{eq:XV:step}
    X^N_t \coloneqq \sum_{i=1}^{N} x_i(t) \chi_{[\theta_{i-1}, \theta_i)}, \qquad \diff{^+}{t}X_t^N =  V^N_t \coloneqq \sum_{i=1}^{N} v_i(t) \chi_{[\theta_{i-1}, \theta_i)}.
  \end{equation}
  Recalling \eqref{eq:psi:part}, we compute
  \begin{equation*}
    \bar{\Psi}^N = \Psi[\bar{X}^N, \bar{V}^N] = \bar{V}^N + (\Phi\ast\bar{\rho}^N)(\bar{X}^N) = \sum_{i=1}^{N} \left[\bar{v}_i + \sum_{j=1}^{N}m_j\Phi(\bar{x}_i-\bar{x}_j) \right] \chi_{[\theta_{i-1}, \theta_i)} = \sum_{i=1}^{N} \bar{\psi}_i \chi_{[\theta_{i-1}, \theta_i)}
  \end{equation*}
  and
  \begin{equation*}
    U[X_t^N] = \bar{\Psi}^N - (\Phi\ast\rho_t^N)(X_t^N) = \sum_{i=1}^{N} \left[\bar{\psi}_i - \sum_{j=1}^{N}m_j\Phi(x_i(t)-x_j(t)) \right] \chi_{[\theta_{i-1}, \theta_i)}
  \end{equation*}
  If we now let \(X_t^N\) satisfy
  \begin{equation*}
    \dot{X}^N_t = U[X_t^N] = \bar{\Psi}^N - (\Phi\ast\rho^N_t)(X^N_t), \quad X_0^N = \bar{X}^N,
  \end{equation*}
  this corresponds exactly to the particle evolution \eqref{eq:EA:part:DE} for the coefficients of the step functions \eqref{eq:XV:step}, which by Lemma \ref{lem:EA:part:unique} is uniquely determined with our inelastic collision rules.
  Here we have by construction \(V^N_0 = \bar{V}^N \in \Hscr_{X^N_0}\).
  Since there are \(N\) discontinuities in the initial \(\bar{X}\), i.e., \(N\) particles, there can at most be \(N-1\) collisions.
  Hence there will be a set of strictly increasing collision times \(\{t_k\}_{k=0}^{N^*+1}\), where for convenience we have included \(t_0 = 0\) and \(t_{N^*+1}=+\infty\), so that \(N^* \le N-1\) is the number of actual collisions.
  If we now consider \(t\in [t_{k},t_{k+1})\), by construction we have the relations
  \begin{equation*}
    \Hscr_{X^N_t} = \Hscr_{X^N_{t_{k}}}, \qquad V_t + \Phi\ast\rho^N_t(X^N_t) = V_{t_k}^N + \Phi\ast\rho^N_{t_k}(X^N_{t_k}) \iff \Psi_{t}^N = \Psi_{t_k}^N.
  \end{equation*}
  At the time of collisions \(\{t_k\}_{k=1}^{N^*}\), by the continuous trajectory of \(\bs{x}\) and inelastic collision rule \eqref{eq:v:proj} for \(\bs{v}\), we have
  \begin{equation*}
    X^N(t_k+) = X^N(t_k-), \qquad V^N(t_k+) = \Proj{\Hscr_{X^N(t_k)}}V^N(t_k-) \iff \Psi^N(t_k+) = \Proj{\Hscr_{X^N(t_k)}}\Psi^N(t_k-).
  \end{equation*}
  Note that we will occasionally write \(X^N(t_k\pm)\) rather than \(X^N_{t_k\pm}\), etc., to avoid cluttered subscripts.
  Now, since \eqref{eq:sticky:OX} and \eqref{eq:sticky:NX&HX} hold by construction, it follows that \(\Psi^N_t = \Proj{\Hscr_{X^N_t}}\bar{\Psi}^N\).
  It remains to show that this sticky evolution is in fact a Lagrangian solution of the differential inclusion \eqref{eq:EA:DI}, and the proof is similar to those of \cite[Theorem 4.2]{natile2009wasserstein} and \cite[Theorem 5.2]{brenier2013sticky}.
  
  \begin{prp}\label{prp:partsols}
    Let \((\bar{\bs{x}}, \bar{\bs{v}}, \bs{m}) \in \mathrm{int}( \K^N )\times\R^N\times\M^N \), and let \((\bar{X}^N, \bar{V}^N) \in \Kscr_{\bs{m}}\times \Hscr_{\bs{m}}\) be the corresponding step functions defined through \eqref{eq:thetai}--\eqref{eq:Km}.
    Then the corresponding sticky particle solution \(X^N_t\) is a sticky Lagrangian solution of \eqref{eq:EA:DI}. In particular, it satisfies the relations of Proposition \ref{prp:sticky:projection}.
  \end{prp}
  \begin{proof}
    From construction it is clear that the monotonicity property \eqref{eq:sticky:OX} is satisfied.
    We will prove by induction on the collision times that \(X^N_t\) satisfies \eqref{eq:EA:DI},
    which by right-continuity is equivalent to
    \begin{equation}\label{eq:EA:DI:step}
      U[X^N_t] - V_t^N = \bar{\Psi}^N - \Psi^N_t \in \del I_{\Kscr}(X_t^N).
    \end{equation}
    Consider \(t \in [0,t_1)\), then by construction \(\Hscr_{X^N_t} = \Hscr_{\bar{X}^N}\).
    Therefore, \(V_t^N = U[X_t^N]\), or \(\Psi_t^N = \bar{\Psi}^N\), leading to \(0 \in \del I_{\Kscr}(X^N_t)\), which is true.
    
    Assume next that \eqref{eq:EA:DI} is satisfied for \(t \in [t_{k-1}, t_k)\), and let us consider \(t \in [t_k,t_{k+1})\) instead.
    By construction we have \(\Hscr_{X^N_t} = \Hscr_{X^N_{t_k}}\), \(V^N(t_k+) = \Proj{\Hscr_{X^N_t}}V^N(t_k-)\), \(\Psi^N(t_k+) = \Proj{\Hscr_{X^N_t}}\Psi^N(t_k-)\) and by hypothesis
    \begin{equation*}
      U[X^N(t_k)] - V^N(t_k-) = \bar{\Psi}^N - \Psi^N(t_k-) \in \del I_{\Kscr}(X^N_{t_k}).
    \end{equation*}
    On the other hand, \(-V^N(t_k-) \in T_{X^N_{t_k}}\Kscr\), and also \(-\Psi^N(t_k-) \in T_{X^N_{t_k}}\Kscr\), which by \eqref{eq:HX&TX&NX} yields \(\Psi^N(t_k-)-\Psi^N_{t_k} \in \del I_{\Kscr}(X^N_{t_k})\).
    Now, since \(\del I_{\Kscr}(X_{t_k})\) is a cone, we add the previous equations together, use \(\Psi^N_t = \Psi^N_{t_k}\) and the monotonicity \eqref{eq:sticky:NX&HX} to deduce that \eqref{eq:EA:DI:step} holds.
  \end{proof}
  Figure \ref{fig:partcoll} illustrates the idea in the above proof with the help of the six particles and their natural velocities from Figure \ref{fig:discflux}.
  Suppose the particles are initially separated and following trajectories corresponding to their natural velocities. Then at time \(t = t_1\) particles 3, 4 and 5 collide, and \(\psi_i(t_1)\) for \(i\in \{3,4,5\}\) is determined by \eqref{eq:RH:part}.
  The corresponding step functions \(\bar{\Psi}^6\) and \(\Psi^6_{t_1}\) are shown in Figure 5, together with the primitive of their difference.
  Compare this primitive to the function in Figure \ref{fig:NXillust}, and note how we must have \(\bar{\Psi}^6-\Psi^6_{t_1} \in \del I_{\Kscr}(X_{t_1})\); this the barycentric lemma \eqref{eq:Oleinik:part} in an \(L^2\)-setting.
  \begin{figure}
    \includegraphics[width=0.7\linewidth]{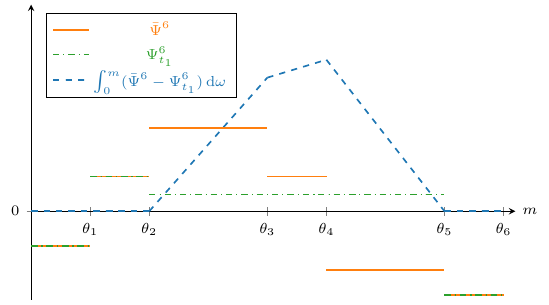}
    \caption{Consider the six particles from Figure \ref{fig:discflux}, where the first collision happens between particles 3, 4 and 5 at time \(t_1\). We have plotted the corresponding step functions \(\bar{\Psi}^6\) and \(\Psi^6_{t_1}\), as well as the primitive of their difference.} 
    
    \label{fig:partcoll}
  \end{figure}
  
  \subsubsection{The globally sticky evolution}
  We now return to the Lagrangian solution from Theorem \ref{thm:LagSol}, and as in \cite[Lemma 5.1]{natile2009wasserstein}, see also \cite[Remarks 5.3, 5.4]{brenier2013sticky}, we want to use the sticky particle evolution above to show that it is globally sticky in the sense of Definition \ref{dfn:LagSol:sticky}.
  To this end we will use the fact that functions of the form \eqref{eq:Hm} are dense in \(\Lp{2}(\Omega)\), and so there are sequences \(\{N_n\}\) and \((\bar{\bs{x}}_n,\bar{\bs{v}}_n,\bar{\bs{m}}_n) \in \K^{N_n}\times\R^{N_n}\times\M^{N_n}\) such that the corresponding \((\bar{X}^{N_n}, \bar{V}^{N_n}) \) converge strongly to \((\bar{X}, \bar{V})\) in \(\Lp{2}(\Omega)\).
  For ease of notation we will then use the superscript \(n\) rather than \(N_n\) in the corresponding quantities.
  The strong convergence and Lemma \ref{lem:PhiConv:unicont} then shows that also \(\Psi[\bar{X}^n,\bar{V}^n]\) converges strongly to \(\Psi[\bar{X},\bar{V}]\).
  On the other hand, assuming only weak convergence for the velocity, we have
  \begin{equation*}
    X^n \to X, \ V^n \rightharpoonup V \ \text{in } \Lp{2}(\Omega) \implies \Psi[X^n,V^n] \rightharpoonup \Psi[X,V] \ \text{in } \Lp{2}(\Omega). 
  \end{equation*}
  Then we have the following result.
  \begin{thm}\label{thm:convergence}
    Let \((X^n_t, V^n_t) \in \Kscr_{\bs{m}}\times\Hscr_{\bs{m}}\), where \(V^n_t \coloneqq \diff{^+}{t}X^n_t\), be sticky particle solutions of \eqref{eq:EA:DI}, as defined in Proposition \ref{prp:partsols}, for which \(\bar{X}_n\) and \(\bar{V}_n\) respectively converge to \(\bar{X}\) and \(\bar{V}\) in \(\Lp{2}(\Omega)\).
    Then:
    \begin{enumerate}[label={\upshape(\alph*)}]
      \item \(X^n_t\) converges to \(X_t\) in \(\Lp{2}(\Omega)\) uniformly in each compact interval, where \(X_t\) is Lipschitz curve with values in \(\Kscr\).
      \item The Lipschitz curve \(X_t\) is a sticky Lagrangian solution of \eqref{eq:EA:DI}, in particular \(X_t\) and \(V_t = \diff{^+}{t}X_t\) satisfy the properties of Corollary \ref{cor:LagSol:props} and Proposition \ref{prp:sticky:projection}.
      \item \(V^n_t\) converges strongly to \(V_t\) in \(\Lp{2}(0,T;\Lp{2}(\Omega))\) for every \(T > 0\).
      \item For any weak accumulation point \(V'\) of \(V^n_t\), we have \(\Proj{\Hscr_{X_t}}V' = V_t\). Similarly, for \(\Psi' \coloneqq V' + \Phi\ast\rho(X_t) \) we have \(\Proj{\Hscr_{X_t}}\Psi' = \Psi_t\).
      \item Let \(\Tcal \subset (0,\infty)\) be the countable set of discontinuities for \(t \mapsto \norm{V_t}_{\Lp{2}(\Omega)}\).
      Then \(V^n_t \to V_t\) and \(\Psi^n_t \to \Psi_t\) in \(\Lp{2}(\Omega)\) for every \(t \in [0,\infty)\setminus\Tcal\).
    \end{enumerate}
  \end{thm}
  \begin{proof}
    Property (a) follows from the stability estimate \eqref{eq:X:stab}; indeed, on every compact time interval \(X^n_t\) is a Cauchy sequence in the closed cone \(\Kscr \subset \Lp{2}(\Omega)\).
    Furthermore, by property (d) of Corollary \ref{cor:LagSol:props}, \(X^n_t\) has Lipschitz-constant bounded by \(\norm{\bar{V}^n}_{\Lp{2}(\Omega)}\).
    From the uniform convergence on each compact and the strong convergence \(\bar{V}^n \to \bar{V}\) it follows that the limit function \(X_t\) has Lipschitz-constant bounded by \(\norm{\bar{V}}_{\Lp{2}(\Omega)}\).
    
    Property (b): That \(X_t\) is a Lagrangian solution follows from stability results for gradient flows in Hilbert spaces; indeed, by assumption it is a weak solution in the sense of \cite[Definition 3.1]{brezis1973operateurs}.
    That it is a Lagrangian solution according to Definition \ref{dfn:gradflow}, i.e., a strong solution in the sense of \cite[Definition 3.1]{brezis1973operateurs}, follows from \cite[Proposition 3.2]{brezis1973operateurs} and \(X_t\) being absolutely continuous on each compact due to (a).
    That \(X_t\) is sticky follows from that \(X^n_t\) is sticky and the stability estimate \eqref{eq:X:stab}, cf.\ \cite[Remark 5.4]{brenier2013sticky}.
    Indeed, from the strong convergence \(\bar{X}^n \to \bar{X}\) it follows that \(\Omega_{\bar{X}} \subset \Omega_{X_t}\).
    By uniqueness, we obtain the same solution \((X_t,V_t)\) by taking \((\bar{X},\bar{V}) = (X_s, V_s)\) for \(0 < s < t\). and evolve this to time \(t-s\); then by the same argument \(\Omega_{X_s} \subset \Omega_{X_t}\).
    
    Property (c) follows directly from the stability estimate \eqref{eq:V:stab}.
    
    Property (d): Let \(n_k\) be an arbitrary subsequence such that \(V^{n_k}_t \rightharpoonup V'\) in \(\Lp{2}(\Omega)\).
    Since the subdifferential of \(\Vcal(X)\) is maximally monotone, its graph is also strongly-weakly closed, meaning we can pass to the limit in \(V^n_t \in U[X_t^n] -\del I_{\Kscr}(X^n_t) \) to obtain \(V' \in U[X_t] - \del I_{\Kscr}(X_t)\).
    Recalling that \(\del I_{\Kscr}(X) \subset \Hscr_{X}^\perp\), we can project the previous relation to find \(\Proj{\Hscr_{X_t}} V' = \Proj{\Hscr_{X_t}}U[X_t]\). Since \(X_t\) is a sticky solution it then follows from \eqref{eq:sticky:V} that \(\Proj{\Hscr_{X_t}} V' = V_t\).
    
    Property (e): Let \(t \in (0,\infty)\setminus\Tcal\), and let \(n_k\), \(V'\) be as in the previous point.
    From the strong convergence in point (c), there exists a dense set \(\Scal \subset (0,\infty)\) we can extract a further subsequence, not relabeled, such that \(V^{n_k}(s) \to V(s)\) for every \(s \in \Scal\).
    In particular, for \(s \in \Scal\) such that \(s < t\) we have
    \begin{equation*}
      \norm{V'}_{\Lp{2}(\Omega)} \le \limsup_{k\to\infty} \norm{V_t^{n_k}}_{\Lp{2}(\Omega)} \le \limsup_{k\to\infty} \norm{V_s^{n_k}}_{\Lp{2}(\Omega)} = \norm{V_s}_{\Lp{2}(\Omega)}.
    \end{equation*}
    Since \(t\) is a point of continuity, we may approach it from below to obtain
    \begin{equation*}
      \norm{V'}_{\Lp{2}(\Omega)} \le \norm{V_t}_{\Lp{2}(\Omega)} \le \norm{U[X_t]-\xi}_{\Lp{2}(\Omega)}, \quad \forall \: \xi \in \del I_{\Kscr}(X_t).
    \end{equation*}
    Then, as \(V_t\) is the unique minimal element of \(U[X_t]-\del I_{\Kscr}(X_t)\), it follows that \(V' = V_t\) and
    \begin{equation*}
      \limsup_{k\to\infty}\norm{V^{n_k}_t}_{\Lp{2}(\Omega)} \le \norm{V_t}_{\Lp{2}(\Omega)},
    \end{equation*}
    which in turn gives the strong convergence of \(V^{n_k}_t\) to \(V_t\). Furthermore, since the subsequence was arbitrary, it follows that the whole sequence \(V^n_t\) converges strongly to \(V_t\).
  \end{proof}
  
  \begin{cor}\label{cor:stickyness}
    Let \(\bar{X} \in \Kscr\) and \(\bar{V} \in \Hscr_{\bar{X}}\). The corresponding Lagrangian solution \(X_t\) of Theorem \ref{thm:LagSol} is globally sticky. In particular, the results of Proposition \ref{prp:sticky:projection} apply to \(X_t\) and its velocity \(V_t\) from Corollary \ref{cor:LagSol:props}.
    Furthermore, the identity \(V_t = v_t \circ X_t\) from \eqref{eq:V:comp} holds for all \(t \ge 0\). 
  \end{cor}
  \begin{proof}
    Since step functions are dense in \(\Lp{2}(\Omega)\), we can  approximate \(\bar{X}\), \(\bar{V}\) with  \(\bar{X}^n\), \(\bar{V}^n\) which converge strongly in \(\Lp{2}(\Omega)\), and apply Theorem \ref{thm:convergence}.
  \end{proof}
  \begin{rem}[Projection formulas]\label{rem:proj:BGSW}
    The attractive Euler--Poisson system, obtained by replacing the velocity-alignment forcing term in \eqref{eq:EA:mom} with the Poisson force \(-\tfrac{\alpha}{2}(\sgn\ast\rho_t)\rho_t\) for \(\alpha > 0\), also has sticky Lagrangian solutions.
    Its corresponding projection formula, cf.\ \cite[Example 6.9]{brenier2013sticky}, is
    \begin{equation*}
      X_t = \Proj{\Kscr} \left(\bar{X} + t \bar{V} - \frac{\alpha}{4}t^2 (2m-1) \right),
    \end{equation*}
    which for \(\alpha = 0\) reduces to the formula for the pressureless Euler system from \cite{natile2009wasserstein} mentioned in Remark \ref{rem:proj:NS}.
    This projection is particularly simple in that the expression being projected can be computed independently of the intermediate values \(X_s\) for \(0 < s < t\); indeed, it is uniquely determined by the initial data.
    This differs from the projection formula \eqref{eq:sticky:X}, which is more similar to the projection formula in \cite[Proposition 3.8]{brenier2013sticky} for a more general sticking force term \(f[\rho_t]\) with Lagrangian representation \(F[X_t]\).
    This aligns with the observation in \cite{leslie2024finite} that one cannot take the (free-flow) particle trajectories of \cite{ha2019complete}, which evolve according to the natural velocities \(\bar{\psi}_i\) and allow for crossings, and project these to recover the sticky particle trajectories.
  \end{rem}
  
  \section{From gradient flows to other solutions}\label{s:other}
  In this section we will use the sticky Lagrangian solution established in the previous section to obtain solutions for other equations.
  Naturally, in the spirit of \cite{natile2009wasserstein,brenier2013sticky}, we will return to the original Euler-alignment system \eqref{eq:EA} and show that the Lagrangian solution provides us with a distributional solution according to Definition \ref{dfn:EA:distsol}.
  
  On the other hand, another main objective of this study is to show that the sticky particle dynamics obtained in \cite{leslie2023sticky} by means of the scalar balance law \eqref{eq:blaw} can be realized from a gradient flow point-of-view.
  To this end, following the ideas of \cite{carrillo2023equiv}, we will see how this balance law can be formally derived from a scalar conservation law, where the flux function is the primitive of the prescribed velocity \(U[X_t]\) in \eqref{eq:U}.
  More importantly, we show that the Lagrangian solution also provides us with an entropy solution of \eqref{eq:blaw}.
  
  \subsection{A distributional solution of the Euler-alignment system}
  Now we return to the Euler-alignment system, recalling the (non-complete) metric space \((\Tscr_2,D_2)\) from \eqref{eq:Tp}, \eqref{eq:Dp}, and propose the following result.
  \begin{thm}\label{thm:EAsol}
    Suppose \((\bar{\rho},\bar{v}) \in \Tscr_2\). Define \(\bar{X}\) according to \eqref{eq:X}, i.e., \(\bar{\rho} = \bar{X}_\#\mathfrak{m}\), and \(\bar{V} = \bar{v}\circ\bar{X}\).
    Then the Lagrangian solution of Theorem \ref{thm:LagSol} provides us with a distributional solution \((\rho_t, v_t) \in \Tscr_2\) of \eqref{eq:EA} in the sense of Definition \ref{dfn:EA:distsol}, with \(\rho_t\) and \(v_t\) given by Corollaries \ref{cor:LagSol:props} and \ref{cor:stickyness}.
    In particular, since the Lagrangian solution is globally sticky, \(S_t \colon (\bar{\rho},\bar{v}) \to (\rho_t,v_t)\) is a semigroup in the metric space  \((\Tscr_2, D_2)\).
  \end{thm}
  \begin{proof}
    By Corollary \ref{cor:stickyness} there is \(v_t \in \Lp{2}(\R,\rho_t)\) such that \(V_t = v_t \circ X_t \) for all \(t \ge 0\).
    In turn, we can define \(\psi_t \in \Lp{2}(\R,\rho_t)\) as in \eqref{eq:psi}.
    We start by showing that \((\rho_t, \psi_t)\) is a distributional solution of \eqref{eq:EAalt}, and proceed as in the proof of \cite[Theorem 3.5]{brenier2013sticky}.
    Given \(\phiv \in C_{\mathrm{c}}^\infty([0,T]\times\R)\) we use \eqref{eq:push} to write
    \begin{align*}
      \int_{0}^{\infty}& \int_{\R} \left[ \del_t \phiv(t,x) \psi_t(t,x) + \del_x \phiv(t,x) v_t(t,x) \psi_t(t,x)\right]\dee\rho_t(x)\dee t \\
      &= \int_{0}^{\infty} \int_{\Omega} \left[ \del_t \phiv(t,X_t(m)) + \del_x \phiv(t,X_t(m)) V_t(m)\right] \Psi_t(m) \dee m\dee t \\
      &= \int_{0}^{\infty} \int_{\Omega} \left[ \diff{^+}{t} \phiv(t,X_t(m)) \right] \bar{\Psi}(m) \dee m\dee t = -\int_{0}^{\infty} \int_{\Omega} \varphi(t,X_t(m)) \diff{}{t} \bar{\Psi}(m) \dee m\dee t = 0,
    \end{align*}
    where we have used \(\bar{\Psi}-\Psi_t \in \del I_{\Kscr}(X_t) \subset \Hscr_{X_t}^\perp\).
    Therefore, the ``momentum'' equation \eqref{eq:EAalt:mom} is satisfied in distributions.
    A similar, even more straightforward, argument shows that the continuity equation \eqref{eq:EAalt:den}, which coincides with \eqref{eq:EA:den}, is satisfied.
    Next, as in the proof of \cite[Theorem 6.3]{leslie2023sticky}, we will make use of that \eqref{eq:EAalt} holds distributionally to show that the momentum equation \eqref{eq:EA:mom} holds in distributions.
    Indeed, we have
    \begin{align*}
      \del_t (\rho_t v_t) + \del_x(\rho_t v_t^2) &= \del_t (\rho_t \psi_t) + \del_x (\rho_t v_t \psi_t) - \del_t(\rho_t (\Phi\ast\rho_t)) - \del_x (\rho_t v_t (\Phi\ast\rho_t)) \\
      &= -(\Phi\ast\rho_t) [\del_t \rho_t + \del_x (\rho_t v_t)] -\rho_t [\del_t (\Phi\ast\rho_t) + v_t \del_x(\Phi\ast\rho_t)] \\
      &= \rho_t (\phi \ast (\rho_t v_t)) - \rho_t v_t (\phi \ast \rho_t).
    \end{align*}
    It remains to verify the initial conditions \eqref{eq:EA:init}.
    The first limit follows from the relation \(\rho_t = X_{t\#}\mathfrak{m}\) and the strong limit \(X_t \to \bar{X}\) in \(\Lp{2}(\Omega)\).
    For the second limit we need to show that
    \begin{equation*}
      \lim\limits_{t\to0+}\int_{\R}\phiv(x)v_t(x)\dee\rho_t(x) = \int_{\R}\phiv(x)\bar{v}(x)\dee\bar{\rho}(x) \quad \text{for every} \ \phiv \in C_\text{b}(\R),
    \end{equation*}
    where \(C_\text{b}(\R)\) is the space of continuous, bounded functions.
    However, this follows from \(\bar{V} = \bar{v}\circ\bar{X} \in \Hscr_{\bar{X}} \subset T_{\bar{X}}\Kscr\), \(V_t = v_t\circ X_t\) and \eqref{eq:V:rightcont}.
    The semigroup property is a consequence of the semigroup property of the sticky Lagrangian solution, cf.\ Remark \ref{rem:semig}.
  \end{proof}
  
  \subsection{An entropy solution of the scalar balance law}
  Let us now see how our notion of gradient flow solutions relate to the entropy solutions studied by Leslie and Tan \cite{leslie2023sticky}.
  
  \subsubsection{Recovering the balance law}
  As noticed in \cite{carrillo2023equiv}, a Lagrangian solution \((X_t, V_t)\) satisfies a conservation law of the form
  \begin{equation}\label{eq:claw}
    \del_t M_t + \del_x \Ucal(t,M_t) = 0,
  \end{equation}
  where the flux function \(\Ucal(t,m)\) is the primitive of the prescribed velocity \(U_t\) satisfying \(U_t - V_t \in \del I_{\Kscr}(X_t) \).
  Define the primitive of \(U[X_t]\) from \eqref{eq:U},
  \begin{equation*}
    \int_0^m U[X_t](\omega)\dee\omega = \int_0^m (\bar{\psi}\circ \bar{X})(\omega)\dee\omega - \int_0^m \int_\Omega \Phi(X_t(\omega)-X_t(\tilde{m}))\dee\tilde{m}\dee\omega \eqqcolon A(m) -S[X_t](m),
  \end{equation*}
  where
  \begin{equation}\label{eq:flux}
    A(m) \coloneqq \int_{0}^{m} \bar{\Psi}(\omega)\dee\omega
  \end{equation}
  corresponds to the flux function \(A\) in \eqref{eq:blaw}, as defined in \cite{leslie2023sticky}.
  The second term evaluated at \(m = M_t(x)\) can be written as
  \begin{equation*}
    S[X_t](M_t(x)) = \int_0^{M_t(x)} \int_\Omega \Phi(X_t(\omega)-y)\dee\rho_t(y)\dee\omega
  \end{equation*}
  Applying the Vol'pert \(BV\)-chain rule, see \cite[Lemma 4.2]{carrillo2023equiv}, and the fact that \(X_t(M_t(x)) = x\) for \(\rho_t\)-a.e.\ \(x\), with \(\rho_t = \del_x M_t\), we obtain
  \begin{equation*}
    \del_x S[X_t](M_t(x)) = \left(\int_\R \Phi(x-y)\dee\rho_t(y)\right)\del_x M_t = (\Phi\ast\rho_t)\del_x M_t
  \end{equation*}
  Combining the above with \eqref{eq:claw} and rearranging the source term \(\del_x S[X_t](M_t)\) we obtain exactly the balance law \eqref{eq:blaw}.
  Note that in \cite{leslie2023sticky} they work with a shifted flux function \(A \colon [-\frac12,\frac12] \to \R\),
  where the lower integral limit in \eqref{eq:flux} is \(-1/2\), since they instead work with the primitive \(\tilde{M}_t = M_t-\frac12\).
  The reason for this, in a sense, more symmetric choice of primitive is to justify the relation \(\Phi\ast\rho_t = \phi\ast M_t\) for their compactly supported \(\rho_t\).
  We will see below that under our assumptions, \(\Phi\ast\rho_t\) is a continuous and linearly bounded function, meaning \((\Phi\ast\rho_t)\rho_t\) is a well-defined measure.
  Hence we prefer to work with this quantity as it is.
  
  \subsubsection{Entropy solutions, the Rankine--Hugoniot and Ole\u{\i}nik E conditions}
  Let \(\eta \colon [0,1] \to \R\) be a Lipschitz and convex function, and suppose \(q \colon [0,1] \to \R\) satisfies \(q' = \eta' A'\).
  Then \((\eta,q)\) is called an entropy-entropy flux pair, and the entropy inequality associated with the balance law \eqref{eq:blaw} is
  \begin{equation}\label{eq:blaw:entropy}
    \del_t \eta(M_t) + \del_x q(M_t) \le (\Phi\ast\rho_t)\del_x\eta(M_t).
  \end{equation}
  
  \begin{dfn}
    A function \(M_t \colon [0,T]\times\R \to [0,1]\) is an entropy solution to \eqref{eq:blaw} if it is nondecreasing, \ and satisfies \eqref{eq:blaw:entropy} in the weak sense for every entropy-entropy flux pair.
  \end{dfn}
  By an approximation argument, one can equivalently require that the inequality holds with the Kru\v{z}kov entropy-entropy flux pair
  \begin{equation}\label{eq:Kpair}
    \eta_k(m) = \abs{m-k}, \qquad q_k(m) = \sgn(m-k)(A(m)-A(k)) = \int_k^m \sgn(\omega-k)\bar{\Psi}(\omega)\dee\omega
  \end{equation}
  for \(k\) in a dense subset of \(\R\), see \cite[Section 3]{golovaty2012existence}.
  We can then, as in \cite{leslie2023sticky}, derive the \textit{Rankine--Hugoniot} and \textit{Ole\u{\i}nik E} conditions.
  Assume \(M_t\) takes the values \(M_t^{-}\) and \(M^{+}_t \) on respectively the left- and right-hand sides of a shock curve \(\{(x,t) \colon x = \sigma(t)\}\) with shock velocity \(\dot{\sigma}(t)\).
  Then \eqref{eq:blaw:entropy} leads to the inequality
  \begin{equation*}
    (\dot{\sigma}(t) + \Phi\ast\rho_t(\sigma(t)))[[\eta(M_t)]] \ge [[q(M_t)]].
  \end{equation*}
  For the choice \((\eta,q)\) = \((\Id, A)\), \eqref{eq:blaw} shows that the above inequality holds as an identity, yielding the Rankine--Hugoniot condition along the shock curve,
  \begin{equation}\label{eq:RH}
    \dot{\sigma}(t) + \Phi\ast\rho_t(\sigma(t)) = \frac{A(M_t^+)-A(M_t^-)}{M_t^+-M_t^-}.
  \end{equation}
  On the other hand, the Kru\v{z}kov entropy-entropy flux pair \eqref{eq:Kpair} yields the Ole\u{\i}nik E condition, cf.\ \cite{oleinik1959uniqueness}, \cite[Section 8.4]{dafermos2016hyperbolic},
  \begin{equation}\label{eq:Oleinik}
    \frac{A(M_t^+)-A(k)}{M_t^+-k} \le \dot{\sigma}(t) + \Phi\ast\rho_t(\sigma(t)) \le \frac{A(k)-A(M_t^-)}{k-M_t^-}, \qquad k \in (M_t^-, M_t^+).
  \end{equation}
  Observe that such a shock would correspond to a maximal interval \((M_t^-, M_t^+) \subset \Omega_{X_t}\) for the corresponding Lagrangian solution \(X_t\), such that \((X_t(m), V_t(m)) = (\sigma(t),\dot{\sigma}(t))\) for \(m \in (M_t^-, M_t^+)\).
  From this point of view, the Rankine--Hugoniot condition \eqref{eq:RH} expresses exactly the identity \(\Psi_t = \Proj{\Hscr_{X_t}}\bar{\Psi}\) for these \(m\).
  Furthermore, by the characterization in Lemma \ref{lem:NX}, the Ole\u{\i}nik E condition \eqref{eq:Oleinik} is equivalent to \(\bar{\Psi}-\Psi_t \in \del I_{\Kscr}(X_t)\) for the same \(m\).
  Then, in analogy with how \(\Psi_t = \Proj{\Hscr_{X_t}}\bar{\Psi}\) can be deduced from \(\bar{\Psi}-\Psi_t \in \del I_{\Kscr}(X_t)\), \eqref{eq:RH} can be deduced from \eqref{eq:Oleinik}.
  We may also recall the inequality \eqref{eq:Oleinik:part} coming from the barycentric lemma in the particle setting; this is exactly \eqref{eq:Oleinik} for the piecewise linear interpolation of \(A\) with interpolation points given by \eqref{eq:thetai}, cf.\ \cite[Section 4]{leslie2023sticky}, \cite[Section 6]{carrillo2023equiv}.
  
  \subsubsection{Verifying the entropy admissibility}
  As mentioned in the introduction, in \cite{leslie2023sticky} the initial data is assumed to be \((\bar{\rho},\bar{v}) \in \Pscr_{\mathrm{c}}(\R)\times\Lp{\infty}(\R,\bar{\rho})\).
  In particular, this ensures that the flux function \(A\) is Lipschitz.
  Moreover, since the solution \((\rho_t,v_t)\) retains these properties, \(\Phi\ast\rho_t\) remains a bounded, continuous function for all \(t \in [0,T]\).
  In turn, the source term \((\Phi\ast\rho_t)\rho_t\) remains a well-defined measure for these times, and the distributional formulation of the balance law \eqref{eq:blaw} makes sense.
  
  We now show that the balance law, in particular the source term, also makes sense for our solution \((\rho_t,v_t) \in \Tscr_2\).
  By the pointwise linear bound on \(\Phi\) in Lemma \ref{lem:Phi:props} we have
  \begin{equation*}
    \abs{\Phi\ast\rho(x)} \le \int_\R \abs{\Phi(x-y)}\dee\rho(y) \le \Phi(1) \left( 1 + \abs{x} + \int_\R \abs{y}\dee\rho(y) \right),
  \end{equation*}
  which shows that the map \(f[\rho] \colon \Pscr_2(\R) \to \Mscr(\R)\) given by \(f[\rho] = (\Phi\ast\rho)\rho\) is pointwise linearly bounded, cf.\ \cite[Definition 6.1]{brenier2013sticky}.
  Moreover, by Lemma \ref{lem:PhiConv:unicont}, we even have that \(f[\rho]\) is uniformly continuous in the sense of \cite[Definition 6.2]{brenier2013sticky}.
  In particular, the source term in \eqref{eq:blaw} is well-defined, and its distributional formulation makes sense.
  
  \begin{thm}\label{thm:blawsol}
    Suppose \((\bar{\rho}, \bar{v}) \in \Tscr_2\), so that \(\bar{\psi} = \bar{v} + \Phi\ast\bar{\rho} \in \Lp{2}(\R,\bar{\rho})\).
    Define \(\bar{M}(x) = \bar{\rho}((-\infty,x])\) and the flux function \(A\).
    Let \(M_t\) be the generalized inverse of the corresponding Lagrangian solution \(X_t\) and recall \(\rho_t = X_{t \#}\mathfrak{m}\).
    Then \(M_t\) is an entropy solution of the scalar balance law \eqref{eq:blaw} with initial value \(\bar{M}\).
  \end{thm}
  
  \begin{proof}
    The initial data is clear from \(\bar{M}\) being the generalized inverse of \(X_0 = \bar{X}\).
    The fact that \(M_t\) satisfies \eqref{eq:blaw} in the weak sense follows from that \((\rho_t,v_t)\) satisfies the continuity equation \eqref{eq:EA:den} and integration by parts, using once more
    \begin{equation*}
      \bar{\psi}\circ \bar{X}-\psi_t\circ X_t = \bar{\Psi}-\Psi_t  \in \del I_{\Kscr}(X_t) \subset \Hscr_{X_t}^\perp.
    \end{equation*}
    In fact, for this part we only use that \(\bar{\Psi}-\Psi_t \in \Hscr_{X_t}^\perp\), equivalent to the Rankine--Hugoniot condition.
    On the other hand, to show that the entropy inequality \eqref{eq:blaw:entropy} holds, we need \(\bar{\Psi}-\Psi_t \in \del I_{\Kscr}(X_t)\), equivalent to the Ole\u{\i}nik E condition.
    
    Let \(M(x) = \rho((-\infty,x])\) for some \(\rho \in \Pscr_2(\R)\), such that the distributional derivative \(\del_x M = \rho\).
    Then for \(f \in H^1(\R)\), the distributional derivative of \(f\circ M \in BV(\R)\) is \(\del_x f(M) = f'_{M}\rho\) for an \(f'_{M} \in \Lp{2}(\R,\rho)\) given by the \(BV\)-chain rule, cf.\ \cite[Lemma 4.2]{carrillo2023equiv}.
    Following \cite[Section 5]{carrillo2023equiv}, for a.e.\ \(m \in \Omega\) we have
    \(\eta_{k,M}'(X_t) = \Proj{\Hscr_{X_t}}(\sgn(m-k))\), \(A_{M}'(X_t) = \Proj{\Hscr_{X_t}}\bar{\Psi} = \Psi_t\), and \(q_{k,M}'(X_t) = \Proj{\Hscr_{X_t}}(\sgn(m-k)\bar{\Psi})\).
    Let \(0 \le \phiv \in C^\infty_{\mathrm{c}}([0,T]\times\R)\) be a smooth test function with compact support; we use this to show that \eqref{eq:blaw:entropy} holds by integrating by parts and using the push-forward relation \eqref{eq:push}. That is,
    using that \(M_t\) is a weak solution, we compute
    \begin{align*}
      &\int_0^T \left( \int_{\R} \left[\eta_k(M_t)\del_t\phiv(t,x) + q_{k}(M_t)\del_x\phiv(t,x)\right]\dee x + \int_{\R}\phiv(t,x)\eta'_{k,M}(M_t)(\Phi\ast\rho_t)\dee\rho_t \right)\dee t  \\
      &\quad= -\int_0^T \int_{\R} \phiv(t,x) \left[q_{k,M}'(t,x) - \eta_{k,M}'(t,x) A_{M}'(t,x) \right]\dee\rho_t\dee t \\
      &\quad= -\int_0^T \int_{\Omega} \phiv(t,X_t) \left[q_{k,M}'(X_t) - \eta_{k,M}'(X_t)A'_{M}(X_t)\right]\dee m\dee t \\
      &\quad= -\int_0^T \int_{\Omega} \phiv(t,X_t) \sgn(m-k) \left[\bar{\Psi} - \Psi_t\right]\dee m\dee t \ge 0,
    \end{align*}
    where the final inequality follows from \eqref{eq:TK1}.
    Indeed, since \(\phiv \ge 0\) it follows from \eqref{eq:TK2} that \(\phiv(t,X_t)\sgn(m-k) \in T_{X_t}\Kscr\) for any time \(t \ge 0\) and \(k \in \R\).
  \end{proof}
  
  \subsubsection*{Possible uniqueness of entropy solutions}
  In \cite{leslie2023sticky} the authors extend the Kru\v{z}kov doubling-of-variables argument to account for the additional source term, thereby proving that entropy solutions for the balance law \eqref{eq:blaw} with initial data \((\bar{\rho},\bar{v}) \in \Pscr_{\mathrm{c}}(\R)\times\Lp{\infty}(\R,\bar{\rho})\) are unique, relying on the fact that the measure remains compactly supported, and that \(\bar{v}\) is essentially bounded, meaning the flux function is Lipschitz.
  
  In our case, the solutions belong to the space \(\Tscr_2\), and the flux function is continuous, but not Lipschitz.
  However, uniqueness of entropy solutions has been proved in \cite{golovaty2012existence} for \textit{conservation} laws with the type of flux function we have here.
  This suggests that the same may be true for balance laws of the form \eqref{eq:blaw}, most likely with some requirements for the source term.
  However, we deem the investigation of this possibility to be outside the scope of our current study.
  
  \section{Clustering formation for sticky dynamics}\label{s:cluster}
  In their follow-up paper \cite{leslie2024finite}, the authors of \cite{leslie2023sticky} derive results on the finite- and infinite-time clustering of their sticky-particle solutions of \eqref{eq:EA}.
  Here clusters refer to the connected components, or maximal intervals, of \(\Omega_{X_t}\).
  Their analysis is based on the flux function \(A\) for the balance law \eqref{eq:blaw} and its lower convex envelope \(A^{**}\).
  It is perhaps not surprising that the flux function \(A\), which encodes the continuum natural velocity \(\bar{\Psi}\), determines the clustering behavior of the Euler-alignment system;
  in \cite{ha2018first} it is shown that for the Cucker--Smale particle dynamics and initial positions \(\bar{\bs{x}} \in \K^N\), there can be finite-time collisions only if the natural velocities are not ordered, i.e., \(\bar{\bs{\psi}} \notin \K^N\), where we recall \eqref{eq:psi:part} and \eqref{eq:cone:part}.
  In our continuum case, recalling \(\Proj{\Kscr}\bar{\Psi} = \diff{^+}{m}A^{**}\) from Proposition \ref{prp:PK}, we observe that any deviation of \(A\) from \(A^{**}\) means that \(\bar{\Psi}\) is not nondecreasing in this region, hence \(\bar{\Psi} \notin \Kscr\), the convex cone from \eqref{eq:K}.
  Indeed, what is called the \textit{supercritical} region \(\Sigma_-\) in \cite{leslie2024finite}, which for our definition \eqref{eq:flux} of \(A\) becomes 
  \begin{equation*}
    \Sigma_- \coloneqq \left\{ m \in \Omega \colon A(m) > A^{**}(m) \right\},
  \end{equation*}
  is exactly where finite-time clustering happens.
  Similarly, their \textit{subcritical} region \(\Sigma_+\)
  is for us
  \begin{equation*}
    \Sigma_+ \coloneqq \left\{ m \in [0,1) \colon A^{**} \text{ is not linear on any interval } [m,m') \right\},
  \end{equation*}
  while the \textit{critical} region \(\Sigma_0\) becomes
  \begin{equation*}
    \Sigma_0 \coloneqq \left\{ \bigcup [m',m'') \colon A^{**} \text{ is linear and equal to } A \text{ on } [m',m'')\right\}.
  \end{equation*}
  Since \(A^{**}\) is convex, \(\diff{^+}{m}A^{**}\) exists for every \(m \in \Omega\) and is a nondecreasing, right-continuous function.
  Therefore, its range \(\ran(\diff{^+}{m}A^{**})\) defines an ordering on \(\Omega\), dividing it into sets of nonincreasing averaged natural velocities.
  Suppose \(m \in \Omega\), and write \(\psi = \Proj{\Kscr}\bar{\Psi}(m)  \in \ran(\diff{^+}{m}A^{**})\).
  Then we write \(L(m) = (\Proj{\Kscr}\bar{\Psi})^{-1}(\psi)\), that is, \(L(m)\) is the preimage of \(\psi\).
  This set then turns out to be analogous to the set \(L(m)\) defined in \cite{leslie2024finite}, where it for \(m \notin \Sigma_+\) is defined as the maximal half-open interval containing \(m\) on which \(A^{**}\) is linear, and \(\{m\}\) otherwise.
  Figure \ref{fig:flux} provides an illustration of these subsets for a generic flux function \(A\).
  As shown in \cite{leslie2024finite}, the sets \(L(m)\) play a central role in the asymptotic behavior.
  
  \begin{figure}
    \includegraphics[width=0.8\linewidth]{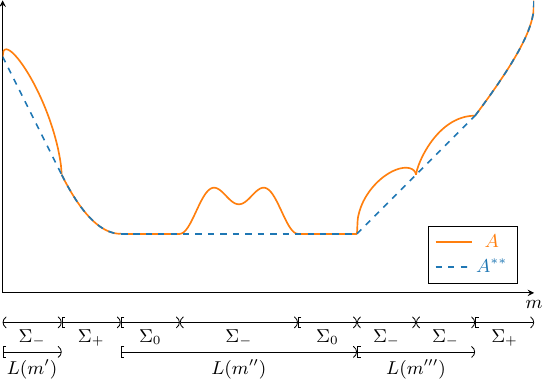}
    \caption{A flux function \(A\) and its lower convex envelope \(A^{**}\). The non-singleton subgroups \(L(m)\) are shown together with the regions \(\Sigma_+\), \(\Sigma_0\) and \(\Sigma_-\).}
    \label{fig:flux}
  \end{figure}
  
  \subsection{Auxiliary estimates}
  From the characterization in Lemma \ref{lem:NX} we know that an element \(\xi \in \del I_{\Kscr}(X)\) can only be nonzero in \(\Omega_{X}\).
  Then for a.e.\ \(m \in \Omega\setminus\Omega_{X_t}\), the differential inclusion \eqref{eq:EA:DI} will be a differential equation since \(\dot{X}_t(m) = U[X_t](m)\).
  In fact, by \eqref{eq:V:PTXU}, the Lagrangian solution satisfies the following differential equation for all \(t \ge 0\),
  \begin{equation}\label{eq:EA:DE}
    \diff{^+}{t}X_t = \Proj{T_{X_t}\Kscr}\bar{\Psi} - \int_{\Omega} \Phi(X_t-X_t(\omega))\dee\omega.
  \end{equation}
  Since our solutions are globally sticky, we have \(\Proj{T_{X_t}\Kscr}\bar{\Psi} = \Proj{\Hscr_{X_t}}\bar{\Psi}\), and by \eqref{eq:TK2} and \eqref{eq:sticky:OX}
  \begin{equation*}
    \Kscr \subset T_{X_t}\Kscr \subset T_{X_s}\Kscr \quad \text{for} \ 0 \le s \le t.
  \end{equation*}
  In particular, since \(X_0 = \bar{X}\) and \(V_0 = \bar{V} \in \Hscr_{\bar{X}}\), \(X_t\) initially evolves according to the continuum version of \eqref{eq:EA:part:DE}.
  Consider \([\alpha_i,\beta_i) \subset \Omega\) for \(i = 1,2\), where \(\beta_1 \le \alpha_2\) such that the intervals do not overlap.
  Then we introduce the averaged quantities
  \begin{equation}\label{eq:Yavg}
    Y^i_t = \dint_{\alpha_i}^{\beta_i} X_t(\omega)\dee\omega, \qquad \Gamma^i_t \coloneqq \dint_{\alpha_i}^{\beta_i} \Proj{T_{X_t}\Kscr}\bar{\Psi}(\omega)\dee \omega,
  \end{equation}
  where \(Y_t^1 \le Y_t^2\) follows from \(X_t \in \Kscr\).
  Then we can derive the following inequalities from \eqref{eq:EA:DE},
  \begin{subequations}\label{eq:EA:DE:LU}
    \begin{align}
      \diff{^+}{t} \left(Y^2_t-Y^1_t\right) & \ge \Gamma^2_t-\Gamma^1_t  - 2\Phi\left(\frac12 \left(Y^2_t-Y^1_t\right)\right), \label{eq:EA:DE:L} \\
      \diff{^+}{t} \left(Y^2_t-Y^1_t\right) & \le \Gamma^2_t-\Gamma^1_t. \label{eq:EA:DE:U}
    \end{align}
  \end{subequations}
  The inequality \eqref{eq:EA:DE:U} follows easily from \(X_t \in \Kscr\) and \(\Phi\) being a nondecreasing function.
  Indeed, we have \(J_t \ge 0\), where
  \begin{equation*}
    J_t \coloneqq \dint_{\alpha_2}^{\beta_2}\int_{\Omega}\Phi(X_t(m)-X_t(\omega))\dee\omega\dee m-\dint_{\alpha_1}^{\beta_1}\int_{\Omega}\Phi(X_t(m)-X_t(\omega))\dee\omega\dee m.
  \end{equation*}
  In order to derive \eqref{eq:EA:DE:L}, we estimate \(J_t\) from above as in \eqref{eq:convbnd},
  \begin{align*}
    J_t &= \dint_{\alpha_2}^{\beta_2} \dint_{\alpha_1}^{\beta_1} \int_{\Omega} \int_{X_t(\tilde{m})}^{X_t(m)} \phi(y-X_t(\omega))\dee y \dee\omega \dee\tilde{m}\dee m \\
    &\le \dint_{\alpha_2}^{\beta_2} \dint_{\alpha_1}^{\beta_1} \int_{X_t(\tilde{m})}^{X_t(m)} \phi\left(y-\frac{X_t(m)+X_t(\tilde{m})}{2}\right)\dee y \dee\tilde{m}\dee m \\
    &= 2\dint_{\alpha_2}^{\beta_2} \dint_{\alpha_1}^{\beta_1} \Phi\left(\frac{X_t(m)-X_t(\tilde{m})}{2}\right)\dee\tilde{m}\dee m \le 2 \Phi\left(\frac{Y_t^2-Y_t^1}{2}\right),
  \end{align*}
  where the first inequality follows from \(\phi(x) = \phi(\abs{x})\) being radially nondecreasing, and the second inequality is a consequence of \(\Phi(x)\) being concave for \(x \ge 0\) and Jensen's inequality.
  Observe that the inequalities \eqref{eq:EA:DE:LU} still holds if we replace one or both of the averaged quantities with \(X_t\) in a point.
  
  Although \(\Gamma^2_t -\Gamma_t^1\) in \eqref{eq:EA:DE:LU} depends on time, we will use their relation with \(A\) and \(A^{**}\) to appropriately bound them with time-independent quantities, so that the following result, comparable with \cite[Lemma 3.4]{leslie2024finite}, is applicable.
  
  \begin{lem}\label{lem:dichotomy}
    Suppose \(X_t \colon [0,\infty) \to \Kscr\), with \(Y_t^{i}\) and \(\Gamma^{i}_t\) for \(i \in \{1,2\}\) as in \eqref{eq:Yavg}.
    We consider the following cases:
    \begin{enumerate}[label={\upshape\Roman*.}]
      \item If \(\Gamma^2_t-\Gamma^1_t \ge 2\sigma > 0\), then there is \(\eta > 0\) depending on \(\Phi\) and \(\sigma\) such that
      \begin{equation*}
        Y_t^2-Y_t^1 \ge \min\left\{Y_0^2-Y_0^1 + \sigma t, \eta\right\} \ge \min\left\{Y_0^2-Y_0^1, \sigma t, \eta\right\}.
      \end{equation*}
      \item If \(\Gamma^2_t-\Gamma^1_t \le -\sigma < 0\), then there exists some time \(\tau\), \(\tau \le (Y_0^2-Y_0^1)/\sigma\), such that \(Y_\tau^2 = Y_\tau^1\).
    \end{enumerate}
  \end{lem}

  \begin{proof}
    \textit{Case I:} Let \(\eta > 0\) be such that \(2\Phi(\frac12\abs{x}) \le \sigma\) whenever \(\abs{x} \le \eta\), which is always possible, cf.\ Lemma \ref{lem:Phi:props}.
    In particular, if \(\Phi\) is invertible, we can choose \(\eta =  2\Phi^{-1}(\frac12\sigma)\).
    Now, either \(Y_t^2-Y_t^1 \ge \eta\), or \(Y_t^2-Y_t^1 < \eta\).
    In the latter case, the right-hand side of \eqref{eq:EA:DE:L} is greater than or equal to \(\sigma\).
    Therefore, we may integrate the leftmost inequality of \eqref{eq:EA:DE:L} to obtain the result.
    
    \textit{Case II:} This follows from integration of \eqref{eq:EA:DE:U}.
  \end{proof}
  
  \subsection{Clustering from the Lagrangian solutions}
  Note that \(\Sigma_-\), like \(\Omega_{X_t}\), is an open set in \(\Omega\), and so can be written as a countable union of disjoint open intervals.
  Let us consider one such maximal interval \((m_-, m_+) \subset \Sigma_-\).
  Note that if \(\Omega_{\bar{X}} \cap \Sigma_- \neq \emptyset\), then any maximal interval \((\alpha,\beta) \subset \Omega_{\bar{X}}\) intersecting \((m_-,m_+)\) must necessarily be contained in \((m_-,m_+)\).
  If not, \(\alpha < m_+ < \beta\), say, we have
  \begin{equation*}
    \frac{A(m_+)-A(m_-)}{m_+-m_-} > \frac{A(m_+)-A(m)}{m_+-m}, \quad m_- < m < m_+,
  \end{equation*}
  by definition of \(\Sigma_-\), and in particular this holds for \(m \in (\alpha,m_+)\).
  However, since \(\bar{\Psi} \in \Hscr_{\bar{X}}\), its slope must be constant on \((\alpha,\beta)\), and so the previous inequality implies \(A^{**}(m_+) = A(m_+) > A(\beta) \ge A^{**}(\beta) \), such that
  \begin{equation*}
    \frac{A^{**}(m_+)-A^{**}(m_-)}{m_+-m_-} > \frac{A^{**}(\beta)-A^{**}(m_-)}{\beta-m_-}.
  \end{equation*}
  This contradicts \(A^{**}\) being convex, cf.\ \cite[Lemma 2.2]{leslie2024finite}.
  A similar argument holds for the case \(\alpha < m_- < \beta\).
  This means that if \(m' < m''\) and \(L(m') \neq L(m'')\), we must have \(\bar{X}(m') < \bar{X}(m'')\), which is also the content of \cite[Lemma 3.5]{leslie2024finite}.
  
  Based on our Lagrangian solutions, we now present a result corresponding to \cite[Theorem 1.7]{leslie2024finite}, where steps of our proof are very much inspired by theirs.
  \begin{thm}\label{thm:clustering}
    Suppose \(m', m'' \in \Omega\). If \(\Proj{\Kscr}\bar{\Psi}(m') < \Proj{\Kscr}\bar{\Psi}(m'')\), then there is a time-independent constant \(c > 0\) such that \(X_t(m'')-X_t(m') \ge c > 0\) for all \(t \ge 0\).
    Furthermore, we have the following cases for \(m \in \Omega\).
    \begin{enumerate}[label={\upshape\Roman*}.]
      \item If \(m \in \Sigma_+\), there is no finite- or infinite-time clustering at \(m\).
      \item If \(m \in \Sigma_-\), there is a finite-time cluster at \(m\) for sufficiently large \(t \ge 0\).
      \item \begin{enumerate}[label={\upshape(\roman*)}]
        \item Suppose \(\int_0^1 \frac{1}{\Phi(x)}\dee x = \infty \), including the case \(\phi \equiv 0\). If \([m',m'') \ni m\) is a finite-time cluster, then either \(m \in \Sigma_0\) and \([m',m'')\) is an initial cluster, or \(m \in \Sigma_-\) and \([m',m'') \subset [m_-,m_+)\), where \((m_-,m_+)\) is a connected component of \(\Sigma_-\).
        No other finite-time clusters are possible.
        \item Suppose \(\phi\) is not identically zero.
        If \(m \notin \Sigma_+\) and \(\left\{\bar{X}(\omega) \colon \omega \in L(m) \right\}\) is bounded, there is an infinite-time cluster at \(m\), given by \(L(m)\).
        \item Suppose in the previous case (ii) that \(\int_0^1 \frac{1}{\Phi(x)}\dee x < \infty \). Then \(L(m)\) is a finite-time cluster.
      \end{enumerate}
    \end{enumerate}
  \end{thm}
  \begin{proof}
    Case I is a direct consequence of the first part of the theorem, which we prove next.
    Consider \(m'' \in \Omega\) and write \(\psi'' = \Proj{\Kscr}\bar{\Psi}(m'')\).
    Then \(L(m'')\) may be a singleton \(\{m''\}\) or an interval \([m_-'', m_+'')\), and we write \(m_-'' = \min L(m'')\).
    For any \(m' \in \Omega\) with \(\psi' = \Proj{\Kscr}\bar{\Psi}(m') < \psi''\) we must then have \(L(m') \neq L(m'')\) and \(\bar{X}(m') < \bar{X}(m'')\).
    Assume for the case of contradiction that \(X_t(m_-'') = X_t(m')\) for some \(m' < m_-''\), and let \(t = \tau\) be the first time this happens.
    For \(t \in [0,\tau)\), \(L(m'')\) can then only collide with mass to its right, and since \(A(m_-'') = A^{**}(m_-'')\), \(A^{**}\) is increasing and \(A \ge A^{**}\), it follows that \(\Proj{\Hscr_{X_t}}\bar{\Psi}(m_-'') \ge \Proj{\Kscr}\bar{\Psi}(m_-'') = \psi''\) in this time interval.
    
    Now consider the interval \([m',m_-'')\) and introduce the averaged quantities 
    \begin{equation*}
      Y_t = \dint_{m'}^{m_-''} X_t(\omega)\dee\omega, \qquad \Gamma_t = \dint_{m'}^{m_-''} \Proj{\Hscr_{X_t}} \bar{\Psi}(\omega)\dee\omega,
    \end{equation*}
    such that \(Y_\tau = X_\tau(m_-'')\).
    Let us write \(m_+' = \sup L(m')\), where if \(L(m') = \{m'\}\) we have \(m_+' = m'\), otherwise \(m' < m_+'\).
    Initially, \(\Proj{\Hscr_{X_0}}\bar{\Psi} = \bar{\Psi}\) and we compute
    \begin{align*}
      \dint_{m'}^{m_-''} \bar{\Psi}(\omega)\dee\omega &= \frac{A(m_+') - A(m') + A(m_-'') - A(m_+')}{m_-''-m'} \le \frac{A^{**}(m_+') - A^{**}(m') + A^{**}(m_-'') - A^{**}(m_+')}{m_-''-m'},
    \end{align*}
    where \(A^{**}(m_+')-A^{**}(m')\) and \(A^{**}(m_-'')-A^{**}(m_+')\) may be zero, but not simultaneously.
    When nonzero, we have \(A^{**}(m_+')-A^{**}(m') = (m_+'-m')\psi'\) and \(A^{**}(m_-'')-A^{**}(m_+') = (m_-''-m_+')\tilde{\psi}\),
    where \(\tilde{\psi} \in (\psi',\psi'')\) is the average slope of \(A^{**}\) on \((m_+', m_-'')\).
    In any case, we find that \(\Gamma_0 \le \psi'\), and this remains an upper bound for \(\Gamma_t\) whenever mass within \([m_+',m_-'')\) collides, since the projection \(\Proj{\Hscr_{X_t}}\) can only make the extremal values of \(\bar{\Psi}\) less extreme.
    Furthermore, for \(t \in [0,\tau)\), \([m_+',m_-'')\) can only interact with mass to its left; since \(A^{**}\) is increasing, \(A(m_-') = A^{**}(m_-')\) and \(A \ge A^{**}\), collision with mass to the left of \(L(m')\) can only decrease \(\Proj{\Hscr_{X_t}}\bar{\Psi}\) on \([m',m_-'')\).
    Then, since \(\psi' < \psi''\) it follows from case II of Lemma \ref{lem:dichotomy} that \(X_\tau(m_-'') > Y_\tau'\), which contradicts our assumption.
    In particular, mass labels belonging to distinct \(L(m)\) can never collide.
    If \(L(m) = \{m\}\), then \(\Proj{\Hscr_{X_t}}\bar{\Psi}(m) = \bar{\Psi}(m)\).
    Otherwise, if \(L(m) = [m_-,m_+)\), then for any \(m_1, m_2 \in (m_-,m_+)\)
    \begin{equation*}
      \int_{m_2}^{m_+}\Proj{\Hscr_{X_t}}\bar{\Psi}(\omega)\dee\omega \le \int_{m_-}^{m_+}\Proj{\Hscr_{X_t}}\bar{\Psi}(\omega)\dee\omega = \int_{m_-}^{m_+}\bar{\Psi}(\omega)\dee\omega = \Proj{\Kscr}\bar{\Psi}(m) \le \int_{m_-}^{m_1}\Proj{\Hscr_{X_t}}\bar{\Psi}(\omega)\dee\omega.
    \end{equation*}
    Following the proof of case I in Lemma \ref{lem:dichotomy}, we find that the time-independent constant \(c\) can be chosen as \(c = \min\{\bar{X}(m'')-\bar{X}(m'), \eta\}\) for some \(\eta > 0\) depending on \(\psi''-\psi' > 0\).
    
    \textit{Case II:} Consider \(m \in (m_-, m_+)\), where \((m_-, m_+) \subset \Sigma_-\) is a maximal interval.
    Now, consider any \(m_1, m_2 \in (m_-,m_+)\) where \(\bar{X}(m_1) < \bar{X}(m_2)\).
    We will prove that the mass centers
    \begin{equation*}
      Y_t^1 = \dint_{m_-}^{m_1}X_t(\omega)\dee\omega, \qquad Y_t^2 = \dint_{m_2}^{m_+}X_t(\omega)\dee\omega
    \end{equation*}
    must coincide in finite time, which necessarily means that \(X_t(m_1)\) and \(X_t(m_2)\) must as well.
    As long as \(X_t(m_1) < X_t(m_2)\), by definition of \(\Sigma_-\), we must have
    \begin{equation*}
      \Gamma_t^1 = \dint_{m_-}^{m_1} \Proj{\Hscr_{X_t}}\bar{\Psi}(\omega)\dee\omega \ge \min_{m_1 \le m \le m_2}\frac{A(m)-A(m_-)}{m-m_-} \eqqcolon \Gamma^1 > \frac{A(m_+)-A(m_-)}{m_+-m_-}
    \end{equation*}
    and
    \begin{equation*}
      \Gamma_t^2 = \dint_{m_2}^{m_+} \Proj{\Hscr_{X_t}}\bar{\Psi}(\omega)\dee\omega \le \max_{m_1 \le m \le m_2}\frac{A(m_+)-A(m)}{m_+-m} \eqqcolon \Gamma^2 < \frac{A(m_+)-A(m_-)}{m_+-m_-}.
    \end{equation*}
    The result then follows from case II of Lemma \ref{lem:dichotomy}.
    
    \textit{Case III:} (i):
    From before we know that mass labels belonging to different \(L(m)\) will never cluster, so without loss of generality we consider \(L(m) \neq \{m\}\) with the associated value \(\psi = \Proj{\Kscr}\bar{\Psi}(m)\).
    We note that \(L(m)\) may contain segments belonging to both \(\Sigma_0\) and \(\Sigma_-\).
    If \(m \in \Sigma_-\), we write \(C(m) = [m_+,m_-)\), where \((m_-, m_+) \subset \Sigma_-\) is the maximal interval containing \(m\).
    Alternatively, if \(m \in \Sigma_0\) and there is an initial cluster at \(m\), we denote this cluster by \(C(m)\).
    Otherwise, \(C(m) = \{m\}\).
    We consider \(m',m'' \in L(m)\) with \(m' < m''\) and \(C(m') \neq C(m'')\), and aim to show that these sets cannot cluster in finite time.
    We write \(m_-'' = \min C(m'')\) and \(m_+' = \sup C(m')\).
    Following the same arguments as in the first part of the proof, we deduce that \(\Proj{\Hscr_{X_t}}\bar{\Psi}(m_-'') \ge \psi\); similarly, if \(C(m') \neq \{m'\}\) then \(\dint_{m'}^{m_+''}\Proj{\Hscr_{X_t}}\bar{\Psi}(\omega)\dee\omega \le \psi\), otherwise \(\Proj{\Hscr_{X_t}}\bar{\Psi}(m') = \psi\). 
    Writing \(Y_t = \dint_{m'}^{m_+'}X_t(\omega)\dee\omega\) if \(C(m') \neq\{m'\}\) and \(Y_t = X_t(m')\) otherwise, we define \(l(t) = X_t(m_-'')-Y_t\).
    Observe that \(X_t(m_-'')-Y_t \le X_t(m'')-X_t(m')\),
    so \(l(t)\) gives a lower bound on their distance, and we have \(l(0) > 0\).
    From \eqref{eq:EA:DE:L} it follows that \(\diff{^+}{t}l(t) \ge -2\Phi(l(t)/2)\), which we integrate to obtain 
    \begin{equation*}
      \int_{l(t)}^{l(0)} \frac{\dee x}{\Phi(\frac{x}{2})} \le 2t.
    \end{equation*}
    From our assumption, \(l(t)\) can only become zero in infinite time. 
    
    (ii): Consider \(m \in \Omega\) with \(L(m) = [m_-, m_+)\) and \(\Proj{\Kscr}\bar{\Psi}(m) = \psi\), where we necessarily have \(A(m_\pm) = A^{**}(m_\pm)\).
    Since \(m \mapsto X_t(m)\) is nondecreasing and \(\bar{X}(\omega)\) is bounded for \(\omega \in L(m)\), both \(\lim\limits_{\omega\to m_-+}\bar{X}(m) = \bar{X}(m_-)\) and \(\lim\limits_{\omega\to m_+-}\bar{X}(m) \eqqcolon \bar{X}(m_+-)\) are finite.
    From before we know that it is only possible for \(X_t(m_-)\) to collide with mass to its right, leading to the inequality \(\Proj{\Hscr_{X_t}}\bar{\Psi}(m_-) \ge \Proj{\Kscr}\bar{\Psi}(m_-) = \psi\).
    
    On the other hand, for any \(m' \in L(m)\), the set \([m',m_+)\) can only collide with mass to its left, leading to
    \begin{equation*} \dint_{m'}^{m_+}\Proj{\Hscr_{X_t}}\bar{\Psi}(\omega)\dee\omega \le \dint_{m'}^{m_+}\Proj{\Kscr}\bar{\Psi}(\omega)\dee\omega = \psi.
    \end{equation*}
    We then use \eqref{eq:EA:DE:U} to deduce that for \(t \ge 0\) we have
    \begin{equation*}
      \dint_{m'}^{m_+}X_t(\omega)\dee\omega -X_t(m_-) \le \dint_{m'}^{m_+}\bar{X}(\omega)\dee\omega -\bar{X}(m_-).
    \end{equation*}
    Letting \(m'\) tend to \(m_+\) we find that \(X_t(m_+-)-X_t(m_-) \le \bar{X}(m_+-)-\bar{X}(m_-)\).
    Then, by the oddness and subadditivity of \(\Phi\) from Lemma \ref{lem:Phi:props}, we estimate
    \begin{align*}
      \Phi\ast\rho_t(X_t(m_+-)) - \Phi\ast\rho_t(X_t(m_-)) &\ge \int_{m_-}^{m_+} \left[\Phi(X_t(m_+-)-X_t(\omega)) + \Phi(X_t(\omega)-X_t(m_-))\right]\dee\omega \\
      &\ge (m_+-m_-) \Phi\left(X_t(m_+-)-X_t(m_-)\right).
    \end{align*}
    This leads to
    \begin{equation}\label{eq:(ii):1}
      \diff{^+}{t} \left(X_t(m_+-)-X_t(m_-)\right) \le -(m_+-m_-) \Phi\left(X_t(m_+-)-X_t(m_-)\right),
    \end{equation}
    and since we assumed \(\phi\) not to be identically zero, it follows that \(\Phi(x) > 0\) for \(x > 0\).
    This means that \(l(t) = X_t(m_+-)-X_t(m_-)\) must decrease over time, and we integrate to find
    \begin{equation}\label{eq:(ii):2}
      \int_{l(t)}^{l(0)} \frac{\dee y}{\Phi(y)} \ge (m_+-m_-)t.
    \end{equation}
    Here the right-hand side is increasing, and so it follows that \(\lim\limits_{t\to\infty}l(t) = 0\).
    
    (iii): Applying our additional assumption in \eqref{eq:(ii):2} we obtain an upper bound on the time it takes to achieve \(l(t) = 0\).
  \end{proof}
  
  \subsubsection*{Assumptions on \(\phi\)}
  Comparing Theorem \ref{thm:clustering} to \cite[Theorem 1.7]{leslie2024finite}, there are some slight differences in the assumptions for case III.
  The assumption of bounded \(\phi\), i.e., \(\phi(x) \le \phi(0)\) for some \(\phi(0) \ge 0\) is covered by (i).
  Indeed, then \(\Phi(x) \le \phi(0) x\) for \(x \ge 0\), the reciprocal of \(\Phi\) is not integrable at the origin, and one finds the estimate \(l(t) \ge l(0)\e^{-\phi(0)t}\).
  
  In the subcase (ii) they make use of the bounded support of their measure, ensuring that the image of \(L(m)\) through \(X_t \colon \Omega \to \R\) can be made arbitrarily small.
  For the same reasons, we assume that \(\bar{X}(\omega)\) is bounded for \(\omega \in L(m)\), so that we start with a finite interval.
  To get the lower bound, they assume a heavy tail, \(\int_1^\infty\phi(x)\dee x = \infty\), which implies global communication \(\phi(x)>0\).
  Under this assumption, we could combine it with \(\Phi(x) \ge \phi(x) x\) to find \(\Phi(l(t)) \ge \phi(l(0))l(t)\).
  In turn this leads to \(l(t) \le l(0) \e^{-(m_+-m_-)\phi(l(0)) t}\), which is similar to their bound.
  
  In subcase (iii), the assumption of weakly singular \(\phi\), cf.\ \eqref{eq:w-sing}, is covered by the integrability of the reciprocal of \(\Phi\) at the origin.
  Moreover, with the global communication assumption one has that \(\diff{^+}{t}l(t) \le -(m_+-m_-)\phi(R) l(t)\) as long as \(l(t) \ge R\).
  Then, when \(l(t) < R\), one follows \cite{leslie2024finite} in using \eqref{eq:w-sing} to show finite-time clustering.
  
  \subsection{Tails and flocking of subgroups}
  Case III of Theorem \ref{thm:clustering} featured some assumptions on the singularity of \(\phi\) at the origin, through whether or not the reciprocal of \(\Phi\) is integrable at the origin.
  Here we saw that the mass contained in a non-singleton \(L(m)\) exhibits a kind of ``local flocking'', in the sense that the diameter of \(X_t(L(m))= \{X_t(\omega) \colon \omega \in L(m)\}\) cannot increase.
  Moreover, if \(\phi\) is nonsingular there is no finite-time clustering except for existing clusters or supercritical regions.
  On the other hand, if \(\phi\) is weakly singular, mass in \(L(m)\) will cluster in finite time.
  
  In Theorem \ref{thm:clustering} we have seen that the continuum natural velocity \(\bar{\Psi}\), through its projection on the convex cone \(\Kscr\), partitions \(\Omega\) into subgroups \(L(m)\) of mass which remain uniformly separated.
  Let us see how the tail conditions in \eqref{eq:tail} affect the ``flocking'' of two distinct subgroups \(L(m')\) and \(L(m'')\) with \(\psi' = \Proj{\Kscr}\bar{\Psi}(m') < \Proj{\Kscr}\bar{\Psi}(m'') = \psi''\).
  As we have seen in the proof of Theorem \ref{thm:clustering}, the dynamics of the Euler-alignment system is determined by a competition
  between the difference in natural velocities \(\psi''-\psi'\), which drives \(L(m')\) and \(L(m'')\) apart, and the strength of the communication \(\Phi\) which tends to align them.
  
  First, suppose \(\phi\) is integrable on \(\R\), i.e., has a thin tail.
  Then \(\Phi(x)\) is bounded and we have \(\lim\limits_{x\to\infty} 2\Phi(x) = \norm{\phi}_{\Lp{1}}\).
  In particular, we see that the communication is not strong enough to prohibit the subgroups from drifting apart if the difference in natural velocities is large.
  Indeed, suppose \(c = \psi'' - \psi' - \norm{\phi}_{\Lp{1}} > 0\), and let \(l(t)\) be the distance between the centers of mass for the two subgroups.
  Then it follows from \eqref{eq:EA:DE:L} that
  \(l(t) \ge l(0) + c t\), i.e., the distance increases linearly in time.
  
  On the other hand, if \(\phi\) is not integrable, i.e., has a fat tail, then \(\lim\limits_{x\to\infty} \Phi(x) = \infty\).
  In particular, it is invertible for any \(x \in \R\).
  Then it follows that \(\Phi\) always can become large enough to balance \(\psi''-\psi'\).
  Indeed, suppose for simplicity that \(\bar{X}(L(m'))\) and \(\bar{X}(L(m''))\) are bounded, and write \(m_-' = \min L(m') < \sup L(m'') = m_+''\).
  Then, writing \(l(t) = X_t(m_+''-)-X_t(m_-')\) and following the derivation of \eqref{eq:(ii):1}, we find
  \begin{equation*}
    \diff{^+}{t} l(t) \le \psi'' - \psi' -(m_+''-m_-') \Phi\left(l(t)\right).
  \end{equation*}
  In particular, we see that if \(l(t) > \Phi^{-1}((\psi''-\psi')/(m_+''-m_-'))\), the distance \(l(t)\) between the outer edges of \(L(m')\) and \(L(m'')\) must decrease to this threshold or less.
  On the other hand, from \eqref{eq:EA:DE:L} we see that there is a corresponding lower threshold for the distance between \(L(m')\) and \(L(m'')\) given by \(2 \Phi^{-1}((\psi''-\psi')/2)\).
  
  \subsection*{Acknowledgments}
  This research was supported by the Research Council of Norway through grant no.~286822, \textit{Wave Phenomena and Stability---a shocking combination (WaPheS)}, and by the Swedish Research Council through grant 2021-06594 while the author was in residence at Institut Mittag-Leffler in Djursholm, Sweden during the autumn semester 2023.
  The author is also grateful to the anonymous referees whose feedback helped improve the presentation of this work.

\end{document}